\documentclass[12pt,reqno]{amsart}
\usepackage{amsmath,amssymb,amscd,fancyhdr,mathrsfs,amsxtra,amsthm}
\usepackage[margin=0.8in]{geometry}
\usepackage{graphicx}
\usepackage{color}
\usepackage{bm} 
\usepackage{enumitem} 
\allowdisplaybreaks

\usepackage{esint}

\newcommand{\R}{{\mathbb R}}

\newcommand{\pa}{\partial}
\newcommand{\na}{\nabla}
\newcommand{\T}{\mathbb T}
\newcommand{\di}{\nabla\cdot}
\newcommand{\eps}{\varepsilon}

\newcommand{\ue}{\uu^{\eps}}

\newcommand{\V}{\mathcal V}

\newcommand{\ox}{\omega}

\newcommand{\vre}{\vr^{\eps}}

\newcommand{\Tintt}{\int_\tau^t\int_{\mathbb T^d}}
\newcommand{\Bintt}{\int_{\tau}^t\int_{\Omega_{\eps_2}}}
\newcommand{\Binttt}{\frac{1}{\eps_3}\int_{\eps_1}^{\eps_1+\eps_3}\Bintt}
\newcommand{\BintttB}{\frac{1}{\eps_3}\int_{\eps_1}^{\eps_1+\eps_3}\int_0^t\int_{\pa\Omega_{\eps_2}}}
\newcommand{\dxte}{dxdsd\eps_2}
\newcommand{\dHte}{d\mathcal{H}^{d-1}(\theta)dtd\eps_2}

\newtheorem{theorem}{Theorem}[section]
\newtheorem{lemma}[theorem]{Lemma}

\newtheorem{remark}[theorem]{Remark}

\newtheorem{definition}{Definition}[section]

\newcommand{\vr}{\varrho}
\newcommand{\uu}{\mathbf{u}}
\newcommand{\supp}{\mathrm{supp}}

\newcommand{\qeda}{\hspace{10mm}\hfill $\square$}

\usepackage[colorlinks=true,citecolor=magenta,linkcolor=magenta]{hyperref}

\begin{document}

\title[Energy conservation for Euler equations]{Energy conservation for inhomogeneous incompressible and compressible Euler equations}

\author[Q-H. Nguyen]{Quoc-Hung Nguyen}
\author[P-T. Nguyen]{Phuoc-Tai Nguyen}
\author[B. Q. Tang]{Bao Quoc Tang}

\address{Quoc-Hung Nguyen \hfill\break
	Scuola Normale Superiore, Piazza dei Cavalieri 7, 56100 Pisa, Italy;\hfill\break  Department of Mathematics, New York University Abu Dhabi, Abu Dhabi, United Arab Emirates}
\email{quochung.nguyen@sns.it,qn2@nyu.edu} 

\address{Phuoc-Tai Nguyen \hfill\break
	Department of Mathematics and Statistics, Masaryk University, 61137 Brno, Czech Republic}
\email{ptnguyen@math.muni.cz, nguyenphuoctai.hcmup@gmail.com} 

\address{Bao Quoc Tang \hfill\break
	Institute of Mathematics and Scientific Computing, University of Graz, Heinrichstrasse 36, 8010 Graz, Austria}
\email{quoc.tang@uni-graz.at} 

\date{\today}

\thanks{}

\begin{abstract}
	Energy conservations are studied for inhomogeneous incompressible and compressible Euler equations with general pressure law in a torus or a bounded domain. We provide sufficient conditions for a weak solution to conserve the energy. By exploiting a suitable test function, the spatial regularity for the density is only required to be of order $2/3$ in the incompressible case, and of order $1/3$ in the compressible case. When the density is constant, we recover the existing results for classical incompressible Euler equation. 
%
\end{abstract}



\maketitle

\tableofcontents


\section{Introduction and Main results}
Let $\Omega$ be either $\T^d$ or a bounded and connected domain  in $\R^d$ with $C^2$ boundary $\partial\Omega$, with $d\geq 2$. This paper studies the conservation of energy for weak solutions to the \textit{inhomogeneous incompressible Euler equation} 
\begin{equation}\label{E}\tag{E}
	\begin{cases}
		\partial_t \vr + \di(\vr \uu) = 0, &\text{ in } \Omega \times (0,T),\\
		\partial_t (\vr \uu) + \di(\vr \uu\otimes \uu) + \na P = 0, &\text{ in } \Omega \times (0,T),\\
		\di\uu = 0, &\text{ in } \Omega \times (0,T),\\
		\uu(x,t)\cdot n(x) = 0, &\text{ on } \pa\Omega \times (0,T),\\
		(\vr\uu)(x,0)  = \vr_0(x)\uu_0(x), &\text{ in } \Omega,\\
		\vr(x,0) = \vr_0(x), &\text{ in } \Omega,
	\end{cases}
\end{equation}
as well as the \textit{compressible Euler equation with general pressure law}
\begin{equation}\label{Ec}\tag{Ec}
\begin{cases}
\partial_t \vr + \di(\vr \uu) = 0, &\text{ in } \Omega \times (0,T),\\
\partial_t (\vr \uu) + \di(\vr \uu\otimes \uu) + \na p(\vr) = 0, &\text{ in } \Omega \times (0,T),\\
\uu(x,t)\cdot n(x) = 0, &\text{ on } \pa\Omega \times (0,T),\\
(\vr\uu)(x,0)  = \vr_0(x)\uu_0(x), &\text{ in } \Omega,\\
\vr(x,0) = \vr_0(x), &\text{ in } \Omega,
\end{cases}
\end{equation}
where $T>0$ is the time horizon, $n(x)$ denotes the outward unit normal vector field to the boundary $\partial \Omega$\footnote{Naturally, the boundary condition $\uu \cdot n = 0$ is neglected when $\Omega = \T^d$.}, $\varrho: \Omega \times (0,T) \to \R_+$ is the scalar density of a fluid, and  $\uu: \Omega \times (0,T) \to \R^d$ denotes its velocity. In system \eqref{E}, $P: \Omega \times (0,T) \to \R$ stands for the scalar pressure, while the general pressure law $p: [0,\infty) \to \mathbb R$ in \eqref{Ec} satisfies some conditions, which will be specified later. Note that when $p(\vr) = \vr^\gamma$ for some $\gamma>1$, \eqref{Ec} is the well-known isentropic compressible Euler equation. 

\medskip
In his celebrated paper \cite{Ons49} Onsager conjectured that there is dissipation of energy for homogeneous Euler equation (namely $\varrho \equiv 1$ in \eqref{E}) for weak solutions with low regularity. More precisely, if a weak solution is in $C^{\alpha}$ for $\alpha > 1/3$ then the energy is conserved while 
the energy is dissipated if $\alpha < 1/3$. The first landmark result concerning the loss (or gain) of energy is due to Scheffer \cite{Sch93} in which he proved the existence of a weak solution having compact support both in time and space. This was later also recovered by Shnirelman \cite{Shn97} for the torus $\mathbb T^d$. This direction of research has been greatly pushed forward by a series of works of De Lellis and Sz\'ekelyhidi in e.g. \cite{DS12,DS13,DS14,BDIS}. The Onsager's conjecture has been recently settled by Isett in \cite{Ise18a,Ise18}. 
The other direction, i.e. the conservation of energy, was first proved by Constant-E-Titi \cite{CET94} for the torus $\mathbb T^3$. The case of bounded domains is studied only recently in \cite{BT18, BTW} in the context of H\"older spaces and in \cite{DN18, NN18} in the context of Besov spaces.

Much less works concerning the inhomogeneous incompressible Euler equation \eqref{E} and the compressible equation \eqref{Ec} have been published in the literature and so far only the case of a bounded domain with periodic boundary condition, namely ${\mathbb T}^d$, has been treated (see the recent papers  \cite{LS16,FGGW17, CY,ADSW}). More precisely, in \cite{FGGW17}, Feireisl \textit{et al.} provided sufficient conditions in terms of Besov regularity both in time and space of the density $\varrho$, the velocity $\uu$ and the momentum ${\bf m}=\varrho \uu$  to guarantee the conservation of the energy. Their method relies on the idea in \cite{CET94} and requires also regularity conditions on the pressure. Recently, by using a different approach, Chen and Yu \cite{CY} obtained the energy conservation, under a different set of regularity conditions. Their method has the advantage of dealing with vacuum. However, since the authors used the convolution both in time and space, additional regularity in time for the density $\vr$ is required.
%

\medskip
In this paper, we provide modest sufficient conditions for a weak solution to conserve the energy for both inhomogeneous incompressible Euler equation \eqref{E} and compressible Euler equation \eqref{Ec}. An attempt to avoid the time regularity of the density is to use the test function $\frac{1}{\vre}(\vr\uu)^\eps$ instead of $\ue$, where the convolution is {\it only taken in spatial variable} (see the definition of $\uu^\varepsilon$ and $\varrho^\varepsilon$ in the list of notations). This choice of test functions was used in \cite{LS16} for density-dependent Navier-Stokes equations and in \cite[Remark 4.2]{FGGW17} for Euler equations, and it allows to avoid any regularity in time for the density $\vr$. However, by way of compensation, one has to exclude vacuum, or at least assume that the inverse density is essentially bounded, i.e. $\vr^{-1}\in L^\infty(\Omega\times(0,T))$. In this work, we improve this by using the test function $\frac{\phi_{\eta,\eps}}{\vre}(\vr\uu)^\eps$ where $\phi_{\eta,\eps}$ is the cut-off function which vanishes when $\vre \leq \eta$ (see the definition of $\phi_{\eta,\varepsilon}$ in \eqref{cutoff}). This helps us to weaken the condition $\vr^{-1}\in L^\infty(\Omega\times(0,T))$ to $\vr^{-1} \in L^{r_1}(\Omega\times(0,T))$ for some $r_1>0$, with the price that the density needs to be smoother near vacuum (see e.g. Theorem \ref{torus}). Naturally, in case of no vacuum, these additional regularity assumptions are not any more necessary. Moreover, we are able to reduce the spatial smoothness required for the density. More precisely, $\vr \in L^\infty(0,T;B^{\alpha,\infty}_{\infty}(\T^d))$ (with $\alpha = \frac 23$ for \eqref{E} and $\alpha = \frac 13$ for \eqref{Ec}) instead of $\vr \in L^p(0,T;W^{1,p}(\T^d))$ as in \cite{CY}.
Our techniques are also suitable to deal with both the case of a torus or a bounded domain. \smallskip

To state our results, we need first the definition of weak solutions for \eqref{E} and \eqref{Ec}.
 
\begin{definition} \label{solution} A triple $(\varrho,\uu, P)$ is called a weak solution to \eqref{E} if 

(i) \begin{align*}
\int_0^T \int_{\Omega} (\varrho \partial_t \varphi  + \varrho \uu \nabla \varphi ) dxdt = 0 
\end{align*}
for every test function $\varphi \in C_0^\infty(\Omega \times (0,T))$.

(ii) \begin{align}\label{weak}
\int_0^T \int_{\Omega} (\varrho \uu \cdot \partial_t \psi + \varrho \uu \otimes \uu : \nabla \psi + P \nabla \cdot \psi) dxdt = 0 
\end{align}
for every test vector field $\psi \in C_0^\infty(\Omega \times (0,T))^d$.	

(iii) $\varrho(\cdot,t) \rightharpoonup \varrho_0$ in ${\mathcal D}'(\Omega)$ as $t \to 0$, i.e. 
\begin{equation} \label{ini1}
\lim_{t \to 0}\int_{\Omega} \varrho(x,t) \varphi(x)dx =  \int_{\Omega} \varrho_0(x) \varphi(x)dx 
\end{equation}
for every test function $\varphi \in C_0^\infty(\Omega)$. 

(iv) $(\varrho \uu)(\cdot,t) \rightharpoonup \varrho_0 \uu_0$ in ${\mathcal D}'(\Omega)$ as $t \to 0$, i.e. 
\begin{equation} \label{ini2}
\lim_{t \to 0}\int_{\Omega}  (\varrho \uu)(x,t) \psi(x)dx =  \int_{\Omega} (\varrho_0 \uu_0)(x) \psi(x)dx 
\end{equation}
for every test vector field $\psi \in C_0^\infty(\Omega)^d$. 

\medskip
A weak solution for \eqref{Ec} can be defined in the same way where in \eqref{weak}, the scalar pressure $P$ is replaced by the general pressure law $p(\vr)$.
\end{definition} 
For $\beta>0$, $\delta>0$ and $p \geq 1$, we define, in the case of the torus $\T^d$, the quantity\footnote{Clearly here is a slight abuse of notation since the definition is not a norm but only a seminorm.}
\begin{align*}
\|f\|_{\mathcal{V}^{\beta,p}_{\delta}(\mathbb{T}^d)}:=\sup_{|h|<\delta} |h|^{-\beta}\|f(\cdot+h)-f(\cdot)\|_{L^p(\mathbb{T}^d)},
\end{align*}
and in the case of a bounded domain $\Omega$ the quantity
\begin{align*}
\|f\|_{\mathcal{V}^{\beta,p}_{\delta}(K)}:=\sup_{|h|<\delta} |h|^{-\beta}\|f(\cdot+h)-f(\cdot)\|_{L^p(K)},
\end{align*} 
for any $K\subset\subset \Omega$ with $d(K,\pa\Omega) > 2\delta$. For $T>0$ we denote by
\begin{equation*}
	\|f\|_{L^q(0,T;\V_{\delta}^{\beta,p}(\T^d))} = \left(\int_0^T\|f(t)\|_{\V_{\delta}^{\beta,p}(\T^d)}^qdt\right)^{1/q}, \quad \text{ for } 1\leq q<\infty,
\end{equation*}
\begin{equation*}
	\|f\|_{L^\infty(0,T;\V_{\delta}^{\beta,p}(\T^d))} = \underset{t\in (0,T)}{\mathrm{ess}\sup}\,\|f(t)\|_{\V_{\delta}^{\beta,p}(\T^d)},
\end{equation*}
and similarly for $\|f\|_{L^q(0,T;\V_{\delta}^{\beta,p}(K))}$ with $K \subset\subset \Omega$.

For simplicity, we denote by $Q_T = \Omega\times (0,T)$ both for $\Omega$ to be $\T^d$ or a bounded domain with boundary, as it should not lead to any confusion. 
\medskip

Our first main result reads as follows.
\begin{theorem}[Conservation of energy for \eqref{E} in a torus]\label{torus}
	Let $(\vr, \uu, P)$ be a weak solution to \eqref{E} in the case $\Omega = \T^d$. Assume that 
	\begin{equation}  \label{Tu}
	\uu \in L^3(Q_T) \cap L^3(0,T;\V_{\delta_0}^{\frac{1}{3},3}(\T^d)), \quad \limsup_{\delta\to 0}\|\uu\|_{L^3(0,T;\V_{\delta}^{\frac{1}{3},3}(\mathbb{T}^d))}= 0, \quad P \in L^{\frac{3}{2}}(Q_T), 
	\end{equation}
	\begin{equation}\label{Trho1}
	0 \leq	\vr \in L^\infty(Q_T) \cap  L^\infty(0,T;\V_{\delta_0}^{\frac{2}{3},\infty}(\T^d)),
	\end{equation}
	\begin{align} \label{Trho2}
	\varrho^{-1} \in L^{r_1}(Q_T) \quad \text{and} \quad {\bf 1}_{ \{ \varrho \leq \delta_0\} }\partial_t \varrho   \in L^{r_2}(Q_T),
	\end{align}	
	for some $\delta_0>0$ and $r_1>0, r_2 \geq 3$ such that $\frac{1}{r_1} + \frac{1}{r_2} = \frac{1}{3}$.

	Then the energy for \eqref{E} conserves for all time, i.e.
	\begin{equation} \label{conservation1}
		\int_{\Omega}(\vr|\uu|^2)(x,t)dx = \int_{\Omega}(\vr_0|\uu_0|^2)(x) dx \quad  \quad \forall t\in (0,T).
	\end{equation}
\end{theorem}
\begin{remark}\label{remark1}\hfill
\begin{itemize}
	\item In the case of no vacuum, i.e. $\vr(x,t) \geq \underline{\vr} > 0$, then \eqref{Trho2} is obviously satisfied by choosing $\delta_0 < \underline{\vr}$, and therefore we require only the assumption \eqref{Trho1} on the density.
	
	
	\item Though we allow the density $\vr$ to be zero at some places, the condition $\|\varrho^{-1}\|_{L^{r_1}(Q_T)} <+\infty$ in particular prevents $\vr$ to be zero in a set with positive measure. Consequently, we have the following useful limit
	\begin{equation}\label{zero_level}
		\limsup_{\eta\to 0}|\{(x,t)\in Q_T:\, \vr(x,t) \leq \eta\}| = 0.
	\end{equation}
	\item It is easy to check that if $f\in W^{\frac 13,3}(\T^d)$, then $f\in\V_{\delta}^{\frac{1}{3},3}(\T^d)$ for all $\delta>0$ and $$\limsup_{\delta\to 0}\|f\|_{\V_{\delta}^{\frac{1}{3},3}(\mathbb{T}^d)}= 0.$$
	Thus, condition (4) holds if $\uu\in L^3(0,T;W^{\frac 13,3}(\T^d))$ and $P\in L^{\frac 32}(Q_T)$.
	\item When $\vr \equiv {\normalfont const}$, we recover the classical assumption $\uu \in L^3(0,T;W^{\frac 13,3}(\T^d))$ for the homogeneous incompressible Euler equation (see e.g. \cite{CET94}). Note that the condition $P\in L^{\frac 32}(Q_T)$ is only needed when the density is not constant.
\end{itemize}
\end{remark}
Our results in Theorem \ref{torus} improve that of \cite{CY} where the authors assume in particular $\vr \in L^p(0,T;W^{1,p}(\T^d))$ and $\uu \in L^q(0,T;B_q^{\alpha,\infty}(\T^d))$ with $\frac 1p + \frac 3q \leq 1$ and $\alpha > \frac 13$. Here we are able to reduce the spatial regularity of the density $\vr$ to the order $\frac 23$ and keep the original order $\frac 13$ of the velocity. When $\|\uu\|_{L^3(0,T;\V_\delta^{\alpha,3}(\T^d))} < \infty$ for some $\alpha > \frac 13$, the limit condition in \eqref{Tu} is automatically satisfied. It's remarked that \cite{CY} also studies the case where $\vr$ only belongs to $L^\infty((0,T)\times \T^d)$ (where vacuum is allowed to happen), but they need to impose Besov regularity for the velocity both in time and space $\uu \in B_p^{\beta,\infty}(0,T;B_q^{\alpha,\infty}(\T^d))$ with $2\alpha + \beta > 1$ and $\alpha + 2\beta > 1$.
	
	\medskip
	In comparison to \cite{LS16} or \cite[Remark 4.2]{FGGW17}, where the author also used the test function $\frac{1}{\vre}(\vr\uu)^\eps$, we are able to allow vacuum, though not "too much", i.e. $\vr^{-1}\in L^{r_1}(Q_T)$, by choosing the test function $\frac{\phi_{\eta,\eps}}{\vre}(\vr\uu)^\eps$ where the cut-off $\phi_{\eta,\eps}$ vanishes when $\vre \leq \eta$. The price to pay is that $\partial_t \vr$ near vacuum needs to be $L^{r_2}$-integrable. An integrability of the inverse of the density is also imposed in the recent work \cite{ADSW}, in which the authors again require time regularity of both the density and the velocity.
	
	\medskip
For the case of a bounded domain with smooth boundary $\Omega$, we need additionally some behavior of $\uu$ and $P$ near the boundary.  In the sequel, $\Omega_r:=\left\{x\in\Omega:d(x,\partial\Omega)>r\right\}$ for any $r\geq 0$ and $\fint_{E}fdx:=\frac{1}{\mathcal{L}^d(E)}\int_E
fdx$ for any Borel set $E\subset\mathbb{R}^d$. Since  $\Omega$  is a bounded, connected domain with $C^2$ boundary,  we find $r_0>0$ and a unique $C_b^1$-vector function $n:\Omega\backslash\Omega_{r_0}\to S^{d-1}$  such that the following holds true: for any $r\in [0,r_0)$, $x\in \Omega_{r}\backslash \Omega_{r_0}$ there exists  a unique $x_r\in\partial\Omega_{r}$  such that  $d(x,\partial \Omega_{r})=|x-x_r|$ and $n(x)$ is the outward  unit normal vector field to the boundary $\partial\Omega_r$ at $x_r$.  
\begin{theorem}[Conservation of energy for \eqref{E} in a bounded domain]\label{bounded}
	Let $\Omega\subset\mathbb R^d$ be a bounded domain with $C^2$ boundary $\pa\Omega$.  Let $(\vr, \uu, P)$ be a weak solution to \eqref{E}. Assume that 
	\begin{equation}\label{Trho'}
0\leq	\vr \in L^\infty(Q_T),\quad \uu \in L^3(Q_T),\quad 	P \in L^{\frac{3}{2}}(Q_T),
	\end{equation}
	\begin{equation}\label{Tu'}
 \|\varrho\|_{L^\infty(0,T;\V_{\delta}^{\frac{2}{3},\infty}(\Omega_{2\delta}))}+\|\uu\|_{L^3(0,T;\V_{\delta}^{\frac{1}{3},3}(\Omega_{2\delta}))}<\infty, \quad \limsup_{\varepsilon\to 0}\|\uu\|_{L^3(0,T;\V_{\varepsilon}^{\frac{1}{3},3}(\Omega_{\delta}))} = 0\quad\forall\delta>0,
	\end{equation}
	\begin{equation}\label{vrbd}
		\vr^{-1}\in L^{r_1}(Q_T) \quad \text{ and } \quad  \mathbf{1}_{\{\vr\leq \delta_0\}}\partial_t \vr  \in L^{r_2}(Q_T),
	\end{equation}
	for some $\delta_0>0$ and $r_1, r_2>0$ such that $\frac{1}{r_1} + \frac{1}{r_2} = \frac 13$. Moreover, we assume the following boundary layer conditions
	\begin{equation}\label{uboundary}
		\left(\int_0^T\fint_{\Omega\backslash\Omega_{\eps}}| \uu(x,t)|^{3}dxdt\right)^{\frac{2}{3}}\left(\int_0^T\fint_{\Omega\backslash\Omega_{\eps}}|\uu(x,t)\cdot n(x)|^{3}dxdt\right)^{\frac{1}{3}} = o(1) \quad\text{ as } \quad \eps \to 0,
	\end{equation}
	\begin{equation}\label{Pboundary}
		\left(\int_0^T\fint_{\Omega\backslash\Omega_{\eps}}|P(x,t)|^{\frac{3}{2}}dxdt\right)^{2/3}\left(\int_0^T\fint_{\Omega\backslash\Omega_{\eps}}|\uu(x,t)\cdot n(x)|^{3}dxdt\right)^{\frac{1}{3}} = o(1) \quad\text{ as } \quad \eps \to 0.
	\end{equation}	
	Then the energy for \eqref{E} conserves for all time, i.e.
	\begin{equation} \label{conservation2}
		\int_{\Omega}(\vr|\uu|^2)(x,t)dx = \int_{\Omega}(\vr_0|\uu_0|^2)(x)dx  \quad  \forall t\in (0,T).
	\end{equation}
\end{theorem}

\begin{remark}\hfill\\
\begin{itemize}
\item Note that for the case of bounded domain, we only require the density and velocity belong {\normalfont locally} to a Besov space.
\item When $\vr \equiv const$ we recover (and slightly improved) the recent results in \cite{NN18} with the remark that the integrability of the pressure $P$ in \eqref{Trho'} is only needed in the inhomogeneous incompressible case. 
\item Conditions \eqref{uboundary} and \eqref{Pboundary} can be replaced by the following conditions
$$  \limsup_{\varepsilon \to 0}\int_{0}^{T}\fint_{\Omega\backslash \Omega_\varepsilon}|\uu(x,t)|^3dxdt +  \limsup_{\varepsilon \to 0}\int_{0}^{T}\fint_{\Omega\backslash \Omega_\varepsilon}|P(x,t)|^{\frac{3}{2}}dxdt < \infty
$$ 
and 
$$  \liminf_{\varepsilon \to 0}\int_{0}^{T}\fint_{\Omega\backslash \Omega_\varepsilon}|\uu(x,t)\cdot n(x)|^3dxdt = 0. $$ \smallskip

\item As  \cite{NN18}, we put $v(x,t)=\uu(x,t) \cdot n(x)$ for any $x\in \Omega\backslash\Omega_{r_0}$. If the function 
$$\phi : \varepsilon \mapsto \phi(\varepsilon)=\| v \|_{L^3((\Omega\backslash\Omega_{\varepsilon}) \times (0,T) )}$$
satisfies $\phi(\varepsilon) \leq C\varepsilon^{2/3}$ for every $\varepsilon \in (0,r_0)$ with some  $C>0$, then conditions \eqref{uboundary} and \eqref{Pboundary} are fulfilled. \smallskip

\item The conditions in \eqref{Tu'} are satisfied if for each $\delta>0$, $\vr \in L^\infty(0,T;W^{\frac 13,\infty}(\Omega_{2\delta}))$ and $\uu \in L^3(0,T;W^{\frac 13,3}(\Omega_{2\delta}))$.
\end{itemize}
\end{remark}

We now turn to the energy conservation for the compressible Euler equation \eqref{Ec}. We will assume the following conditions of the pressure law $p: (0,\infty) \to \mathbb R$,
\begin{equation}\label{pressure_law}
	p\in C([0,\infty))\cap C^2((0,\infty)) \qquad \text{ and } \qquad 
		\limsup_{z \to 0}\frac{|p(z)| + |p_2(z)|}{|z|} < +\infty,
\end{equation}
	where
	\begin{equation}\label{p2}
		p_2(z) = z\int_1^z\frac{p(\ell)}{\ell^2}d\ell.
	\end{equation}
We remark that, in the isentropic case, i.e. $p(\vr)=\vr^\gamma$ for $\gamma>1$, the conditions in \eqref{pressure_law} are obviously satisfied.
\begin{theorem}[Conservation of energy for \eqref{Ec} in a torus]\label{toruscom}
	Let $(\vr, \uu)$ be a weak solution to \eqref{Ec} in the case $\Omega = \T^d$.  Assume that the pressure law satisfies \eqref{pressure_law}, and the density $\vr$ and the velocity satisfy
	\begin{equation}\label{Trhocom}
	0\leq \vr \in L^\infty(Q_T),\quad  \uu\in L^3(Q_T),
	\end{equation}
	\begin{equation}\label{Tucom}
	\|\varrho\|_{L^{\infty}(0,T;\V_{\delta_0}^{\frac 13,\infty}(\mathbb{T}^d))}+\|\uu\|_{L^3(0,T;\V_{\delta_0}^{\frac{1}{3},3}(\mathbb{T}^d))}<\infty, \quad \limsup_{\delta\to 0}(\|\uu\|_{L^3(0,T;\V_{\delta}^{\frac{1}{3},3}(\mathbb{T}^d))}+\|\varrho\|_{L^{\infty}(0,T;\V_{\delta}^{\frac 13,\infty}(\mathbb{T}^d))})= 0,
	\end{equation}
	\begin{equation}\label{assumption-1}
		\mathbf{1}_{\{\vr \leq \delta_0 \}}\partial_t\vr\in L^3(Q_T),\; \mathbf{1}_{\{\vr \leq \delta_0 \}}\na \vr\in L^\infty(Q_T) \;\; \text{ and } \quad \lim_{\eta\to 0}|\{(x,t)\in Q_T:\; \vr(x,t) \leq \eta\}| = 0,
	\end{equation}	
	for some $\delta_0>0$. 
Then the energy for \eqref{Ec} conserves for all time, i.e.
	\begin{align}\label{conserEc}
	\int_{\mathbb{T}^d}\left(\frac{1}{2}(\varrho|\uu|^2)(x,t)+p_2(\vr)(x,t)\right)dx= \int_{\mathbb{T}^d}\left(\frac{1}{2}(\varrho_0|\uu_0|^2)(x)+p_2(\vr_0)(x)\right)dx \quad \forall t\in (0,T). 
	\end{align}
\end{theorem}

\begin{remark}\label{remark2}
	The conditions \eqref{Trhocom} and \eqref{Tucom} are satisfied if
	\begin{equation*}
		\vr \in L^\infty(0,T;W^{\frac 13,\infty}(\T^d)) \quad \text{ and } \quad \uu \in L^3(0,T;W^{\frac 13,3}(\T^d)).
	\end{equation*}
\end{remark}

In comparison to the incompressible case in Theorem \ref{torus}, we require in \eqref{Tucom} the spatial regularity of $\vr$ to be \textit{only $\frac 13$ instead of $\frac 23$}. This is due to the fact that the general pressure $p$ depends on the density, which allows a finer analysis using smoothness of $\vr$ itself (see estimate \eqref{G1}). We are also able to eliminate any condition on $\vr^{-1}$, though we have to impose the spatial smoothness of $\vr$ near vacuum in \eqref{assumption-1}. Yet, we still need to assume that the vacuum set has zero measure. The result of Theorem \ref{toruscom} is applicable for a large class of pressure law satisfying \eqref{pressure_law}, which includes the isentropic case $p(\vr) = \vr^\gamma$, $\gamma>1$, as a special situation.

\medskip
We have the following parallel result for the case of bounded domains.
\begin{theorem}[Conservation of energy for \eqref{Ec} in a bounded domain]\label{boundedcom}
	Let $\Omega\subset\mathbb R^d$ be a bounded domain with $C^2$ boundary $\pa\Omega$.  Let $(\vr, \uu)$ be a weak solution to \eqref{Ec}. Assume that the pressure law satisfies \eqref{pressure_law}, and the density and the velocity satisfy
	\begin{equation}\label{Trho'com}
	0\leq \vr \in L^\infty(Q_T),\quad \uu \in L^3(Q_T),
	\end{equation}
	and for some $\delta_0>0$,
	\begin{equation}\label{Tu'com}
	 \|\varrho\|_{L^\infty(0,T;\V_{\delta}^{\frac 13,\infty}(\Omega_{2\delta}))}+	 \|\uu\|_{L^3(0,T;\V_{\delta}^{\frac{1}{3},3}(\Omega_{2\delta}))}<\infty, \quad \limsup_{\varepsilon\to 0}\|\uu\|_{L^3(0,T;\V_{\varepsilon}^{\frac{1}{3},3}(\Omega_{\delta}))} + \|\varrho\|_{L^\infty(0,T;\V_{\varepsilon}^{\frac 13,\infty}(\Omega_{\delta}))}= 0,
	\end{equation}
	for any $0<\delta<\delta_0$, and 
	\begin{equation}\label{assumption-2}
		\mathbf{1}_{\{\vr \leq \delta_0 \}}\partial_t\vr\in L^3(Q_T),\; \mathbf{1}_{\{\vr \leq \delta_0 \}}\na \vr\in L^\infty(Q_T) \;\; \text{ and } \quad \lim_{\eta\to 0}|\{(x,t)\in Q_T:\; \vr(x,t) \leq \eta\}| = 0.
	\end{equation}
	Moreover,
	\begin{equation}\label{uboundarycom}
	\left(\int_0^T\fint_{\Omega\backslash\Omega_{\eps}}| \uu(x,t)|^{3}dxdt\right)^{\frac{2}{3}}\left(\int_0^T\fint_{\Omega\backslash\Omega_{\eps}}|\uu(x,t)\cdot n(x)|^{3}dxdt\right)^{\frac{1}{3}} = o(1) \quad\text{ as } \quad \eps \to 0,
	\end{equation}
	\begin{equation}\label{Pboundarycom}\int_0^T\fint_{\Omega\backslash\Omega_{\eps}}|\uu(x,t)\cdot n(x)|dxdt= o(1) \quad\text{ as } \quad \eps \to 0.
	\end{equation}	
	Then the energy for \eqref{E} conserves for all time, i.e.
	\begin{align}\label{conserEcboun}
\int_{\Omega}\left(\frac{1}{2}(\varrho|\uu|^2)(x,t)+p_2(\vr)(x,t)\right)dx= \int_{\Omega}\left(\frac{1}{2}(\varrho_0|\uu_0|^2)(x)+p_2(\vr_0)(x)\right)dx \quad \forall t\in (0,T). 
\end{align}
\end{theorem}

\medskip
\textbf{The paper is organized as follows}: The proofs of Theorems \ref{torus}, \ref{bounded},  \ref{toruscom} and \ref{boundedcom} will be presented in the next four Sections respectively. In Appendix \ref{appendix}, we collect some useful estimates which will be used in the proofs.

\medskip
{\bf Notation:}
\begin{itemize}
	\item For convenience, we simply write $\|f\|_{L^p}$ for $\|f\|_{L^p(Q_T)}$ for $1\leq p \leq \infty$, recalling that $Q_T = \Omega\times (0,T)$ where either $\Omega = \T^d$ or $\Omega\subset \mathbb R^d$ is a bounded domain with smooth boundary.
	\item We denote by $C$ a generic constant, whose value can change from line to line or even the same line. Sometimes we write $C(\lambda)$ to emphasize the dependence on $\lambda > 0$.
	\item For simplicity, we write $\limsup\limits_{\varepsilon_n;...;\varepsilon_1 \to 0}$ for $\limsup\limits_{\varepsilon_n \to 0}...\limsup\limits_{\varepsilon_1 \to 0}$.
	\item Let $\omega: \mathbb R^d \to \mathbb R$ be a standard mollifier, i.e. $\omega(x) = c_0e^{-\frac{1}{1-|x|^2}}$ for $|x| < 1$ and $\omega(x) = 0$ for $|x| \geq 1$, where $c_0$ is a constant such that $\int_{\R^d}\omega(x)dx = 1$. For any $\eps > 0$, we define the rescaled mollifier $\omega_\varepsilon(x) = \eps^{-d}\omega(\frac x\eps)$. For any function $f \in L_{loc}^1(\Omega)$, its mollified version is defined as
	\begin{equation*}
	f^\eps(x) = (f\star \omega_\eps)(x) = \int_{\mathbb R^d}f(x-y)\omega_\eps(y)dy, \quad x \in \Omega_\varepsilon,
	\end{equation*}
	where
	$\Omega_\varepsilon = \{x\in\Omega: d(x,\partial\Omega) > \varepsilon \}$. 
	Throughout this paper, we will use the notations $\uu^\varepsilon(x,t)=\int_{\R^d}\uu(x-y,t) \omega_\varepsilon(y)dy$ and $\varrho^\varepsilon(x,t)=\int_{\R^d} \varrho(x-y,t) \omega_\varepsilon(y)dy$.
\end{itemize}

\section{Proof of Theorem \ref{torus}}\label{sec:torus}
In this section we write $\mathbb T^d$ instead of $\Omega$. By smoothing \eqref{E}, we obtain
\begin{equation}\label{Trho_approx}
	\pa_t\vre + \di(\vr \uu)^\eps = 0
\end{equation}
and
\begin{equation}\label{Tu_approx}
	\pa_t(\vr\uu)^\eps + \di(\vr\uu\otimes\uu)^\eps + \na P^\eps = 0.
\end{equation}
for any $0<\varepsilon<1$. 

Let $\phi \in C^\infty(\R_+)$ such that 
\begin{equation}\label{def_phi}
0 \leq \phi \leq 1,\quad \phi=0 \text{ in } [0,1], \quad \phi=1 \text{ in }[2,\infty) \quad \text{ and } \quad |\phi'|+ |\phi''| \leq C \text{ in } [0,\infty),
\end{equation}
for some constant $C$. For $\eta>0$, put \begin{equation}\label{cutoff}
\phi_{\eta,\varepsilon}(x,t)=\phi\left(\frac{\vre(x,t)}{\eta}\right).
\end{equation}
We see that $\phi_{\eta,\varepsilon} = 0$ in $\{\vre\leq \eta\}$, $\phi_{\eta,\varepsilon} = 1$ in $\{\vre \geq 2\eta\}$ and $0 \leq \phi_{\eta,\varepsilon} \leq 1$ in $Q_T$. 

Let $\delta_0$ be the positive constant in \eqref{Tu} and \eqref{Trho2}. Let $M$ be a positive constant such that 
\begin{equation} \label{condnear0}
\|\varrho\|_{L^\infty}+\|\varrho\|_{L^{\infty}(0,T;\V_{\delta_0}^{\frac{2}{3},\infty}(\mathbb{T}^d))} + \| \varrho^{-1} \, {\bf 1}_{ \{ \varrho \leq \delta_0  \} } \|_{L^{r_1}} + || \partial_t \varrho\, {\bf 1}_{ \{\varrho \leq \delta_0  \}}||_{L^{r_2}} \leq M.
\end{equation}	

Take $\varepsilon>0$ and $\eta>0$ such that $0<\varepsilon^{\frac{2}{3}} < \frac{1}{3M}\eta <\frac{1}{12M} \delta_0$.   

\medskip
Multiplying \eqref{Tu_approx} by $\phi_{\eta,\varepsilon} \frac{(\vr\uu)^\eps}{\varrho^\varepsilon}$ then integrating on $(\tau,t)\times\mathbb T^d$, for $0<\tau<t<T$, we get
\begin{equation}\label{h1}
	\int_{\tau}^t \int_{\T^d} \phi_{\eta,\varepsilon}\frac{(\vr\uu)^\eps}{\vre}  \pa_t(\vr\uu)^\eps dxds + \int_{\tau}^t \int_{\T^d} \phi_{\eta,\varepsilon}\frac{(\vr\uu)^\eps}{\vre} \di(\vr\uu\otimes\uu)^\eps dxds + \int_{\tau}^t \int_{\T^d} \phi_{\eta,\varepsilon}\frac{(\vr\uu)^\eps}{\vre} \na P^\eps dxds= 0.
\end{equation}
Denote by $(A), (B)$ and $(C)$ the terms on the left-hand side of \eqref{h1} respectively. We will estimate their convergence separately in the next subsections. 

\subsection{Estimate of $(A)$}
We rewrite $(A)$ as, using $\partial_t \vre = \di(\vr\uu)^\eps$,
\begin{equation}\label{h3}
\begin{aligned}
(A) &= \frac 12 \int_{\tau}^t \int_{\T^d} \phi_{\eta,\varepsilon} \pa_t\left(\frac{|(\vr\uu)^\eps|^2}{\vre}\right)dxds - \frac 12 \int_{\tau}^t \int_{\T^d} \frac{ \phi_{\eta,\varepsilon}}{(\vre)^2}\di[(\vr\uu)^\eps - \vre\ue]|(\vr\uu)^\eps|^2dxds\\
&\quad - \frac 12 \int_{\tau}^t \int_{\T^d} \frac{ \phi_{\eta,\varepsilon}}{(\vre)^2}\di(\vre\ue)|(\vr\uu)^\eps|^2dxds\\
&=: (A1) + (A2) + (A3).
\end{aligned}
\end{equation}

By integration by parts, we have
\begin{align} \label{A1}
(A1)&= \frac{1}{2}\partial_t\int_{\tau}^t\int_{\T^d}  \left(\phi_{\eta,\varepsilon} \frac{|(\vr\uu)^\eps|^2}{\vre} \right)(x,s)dxds - \frac{1}{2} \int_{\tau}^t \int_{\T^d} \partial_t \phi_{\eta,\varepsilon} \frac{|(\vr\uu)^\eps|^2}{\vre}dxds \\ \nonumber
&=: (A11) + (A12).
\end{align}
From \eqref{condnear0}, Lemma \ref{phi}, the fact that $\supp\, \partial_t \phi_{\eta,\varepsilon}= \supp\,(\phi_{\eta,\varepsilon}'\partial_t\vre)  \subset \{\eta \leq \vre \leq 2\eta\}$ and H\"older inequality, we obtain
\begin{equation} \label{A12-1} \begin{aligned}
|(A12)| &\leq \frac{1}{2}\eta^{-1}\int_\tau^t \int_{\T^d} \left| \phi' \left( \frac{\vre}{\eta} \right)\partial_t \vre  \right| \frac{|(\varrho \uu)^\varepsilon|^2}{\varrho^\varepsilon}dxds \\
&\leq C \eta^{-2}\int_\tau^t\int_{\T^d}|\partial_t\vre||(\varrho \uu)^\varepsilon|^2{\bf 1}_{\{ \eta \leq \vre \leq 2\eta  \}}dxds\\
&\leq C\eta^{-2}\|\partial_t\vre\mathbf{1}_{  \{\eta \leq \varrho^{\varepsilon} \leq  2\eta \}  }\|_{L^3}\|(\vr \uu)^\eps {\bf 1}_{\{ \eta \leq \vre \leq 2\eta  \}}\|_{L^3}^2.
\end{aligned} \end{equation}
By H\"older inequality,
\begin{align*}
|\partial_t \varrho^{\varepsilon}(x,t) | &= \left| \int_{  |y| < \varepsilon } \partial_t \varrho(x-y,t) \omega_\varepsilon(y)dy \right| \\
&\leq \left( \int_{  |y| < \varepsilon } |\partial_t \varrho(x-y)|^3 dy \right)^{\frac{1}{3}} \left( \int_{  |y| < \varepsilon  } \omega_\varepsilon(y)^{\frac{3}{2}} dy \right)^{\frac{2}{3}} \\
&\leq C\varepsilon^{-\frac{d}{3}}  \left( \int_{  |y| < \varepsilon } |\partial_t \varrho(x-y)|^3 dy \right)^{\frac{1}{3}}.
\end{align*}
This, together with Lemma \ref{phi} and \eqref{condnear0}, implies
\begin{equation} \label{A12-3} \begin{aligned}
\int_{\T^d}|\partial_t \varrho^{\varepsilon}(x,t) \mathbf{1}_{  \{\eta \leq \varrho^{\varepsilon} \leq  2\eta \}  }(x,t) |^3 dx &\leq C\varepsilon^{-d} \int_{|y|<\varepsilon} \int_{\T^d} |\partial_t \varrho(x-y,t) \mathbf{1}_{  \{\eta \leq \varrho^{\varepsilon} \leq  2\eta \}  }(x,t) |^3 dx dy \\
&\leq  C\varepsilon^{-d} \int_{|y|<\varepsilon} \int_{\T^d} |\partial_t \varrho(x-y,t) \mathbf{1}_{  \{\frac{1}{3}\eta \leq \varrho \leq  \frac{8}{3}\eta \}  }(x-y,t) |^3 dx dy 
 \\
&=  C\varepsilon^{-d} \int_{|y|<\varepsilon} \int_{\T^d} |\partial_t \varrho(x,t) \mathbf{1}_{  \{\frac{1}{3}\eta \leq \varrho \leq  \frac{8}{3}\eta \}  }(x,t) |^3 dx dy \\
&\leq CM^3.
\end{aligned} \end{equation}
{\color{red}}
Also, by H\"older's inequality,
\begin{align*}
|(\varrho \uu)^{\varepsilon}(x,t) | &= \left| \int_{  |y| < \varepsilon }  \varrho(x-y,t) \uu(x-y) \omega_\varepsilon(y)dy \right| \\
&\leq \left( \int_{  |y| < \varepsilon } \varrho(x-y)^3 | \uu(x-y)|^3 dy \right)^{\frac{1}{3}} \left( \int_{  |y| < \varepsilon  } \omega_\varepsilon(y)^{\frac{3}{2}} dy \right)^{\frac{2}{3}} \\
&= C\varepsilon^{-\frac{d}{3}}  \left( \int_{  |y| < \varepsilon } \varrho(x-y)^3 | \uu(x-y)|^3 dy \right)^{\frac{1}{3}}.
\end{align*}
Consequently, by Lemma \ref{phi} (ii),
\begin{equation} \label{A12-4} \begin{aligned}
&\int_{\T^d} |(\varrho \uu)^{\varepsilon}(x,t) \mathbf{1}_{  \{\eta \leq \varrho^{\varepsilon} \leq  2\eta \}  }(x,t)|^3 dx \\
&\leq  C\varepsilon^{-d} \eta^3 \int_{|y|<\varepsilon} \int_{\T^d} |\uu(x-y,t) \mathbf{1}_{  \{\frac{1}{3}\eta \leq \varrho \leq  \frac{8}{3}\eta \}  }(x-y,t) |^3 dx dy
\\
&\leq  C\varepsilon^{-d} \eta^3 \int_{|y|<\varepsilon} \int_{\T^d} |\uu(x-y,t) \mathbf{1}_{  \{\frac{1}{3}\eta \leq \varrho \leq  \frac{8}{3}\eta \}  }(x-y,t) |^3 dx dy \\
&= C \eta^3 \| \uu \mathbf{1}_{  \{\frac{1}{3}\eta \leq \varrho \leq  \frac{8}{3}\eta \} } \|_{L^3(\T^d)}^3.
\end{aligned} \end{equation}
Combining \eqref{A12-1}, \eqref{A12-3} and \eqref{A12-4} implies
\begin{align} \label{A12-2}
|(A12)| \leq C M\| \uu \mathbf{1}_{  \{\frac{1}{3}\eta \leq \varrho \leq  \frac{8}{3}\eta \} } \|_{L^3(\T^d)}^2.
\end{align}
Therefore, by the dominated convergence theorem, we deduce that
$$ \limsup_{\eta;\delta;\varepsilon;\tau \to 0}|(A12)|=0.
$$
The term $(A11)$ is the desired term. The term $(A3)$ will be canceled with the term $(B22)$ when estimating $(B)$ (see \eqref{h13}). Therefore, for $(A)$ it remains to estimate $(A2)$. 

\medskip
By integration by parts and H\"older's inequality,
\begin{equation} \label{hu0}
|(A2)| \leq \frac{1}{2} \int_{\tau}^{t}\left\|\nabla \left(  \phi_{\eta,\varepsilon} \frac{|(\vr\uu)^\eps|^2}{(\vre)^2} \right) \right\|_{L^{\frac{3}{2}}(\mathbb{T}^d)}\|((\vr\uu)^\eps - \vre\ue)\|_{L^3(\mathbb{T}^d)} ds.
\end{equation}
For any $\delta \in (\varepsilon,\delta_0)$, by Lemma \ref{keylemma-Td} (2), we estimate
\begin{equation} \label{hu01}
\|(\vr\uu)^\eps - \vre\ue\|_{L^3(\mathbb{T}^d)} \leq C(M) \varepsilon^{\frac{2}{3}}  \|\uu\|_{\V_{\delta}^{\frac{1}{3},3}(\T^d)}.
\end{equation}
On the other hand
\begin{equation} \label{hu1} \left\|\nabla \left(  \phi_{\eta,\varepsilon} \frac{|(\vr\uu)^\eps|^2}{(\vre)^2} \right) \right\|_{L^{\frac{3}{2}}(\mathbb{T}^d)} \leq \left\|\nabla  \phi_{\eta,\varepsilon} \left(  \frac{|(\vr\uu)^\eps|^2}{(\vre)^2} \right) \right\|_{L^{\frac{3}{2}}(\mathbb{T}^d)} + \left\| \phi_{\eta,\varepsilon} \nabla \left(   \frac{|(\vr\uu)^\eps|^2}{(\vre)^2} \right) \right\|_{L^{\frac{3}{2}}(\mathbb{T}^d)}.
\end{equation}
From  \eqref{condnear0}, Lemma \ref{keylemma-Td} (1) and Lemma \ref{phi}, we see that  
\begin{equation} \label{est-phi1} \begin{aligned}
\left\|\nabla  \phi_{\eta,\varepsilon} \left(  \frac{|(\vr\uu)^\eps|^2}{(\vre)^2} \right) \right\|_{L^{\frac{3}{2}}(\mathbb{T}^d)} &=\eta^{-1}\left\|\phi' \left( \frac{\vre}{\eta} \right) |\nabla \vre|   \frac{|(\vr\uu)^\eps|^2}{(\vre)^2}  \right\|_{L^{\frac{3}{2}}(\mathbb{T}^d)} \\
&\leq C\eta^{-3}\|\nabla \vre\|_{L^\infty(\T^d)}\|(\vr\uu)^\eps\|_{L^3(\T^d)}^2\\
&\leq CM\eta^{-3}\eps^{-\frac 13}\|\vr\|_{\V_\delta^{\frac 23, \infty}(\T^d)}\|\uu\|_{L^3(\T^d)}^2 \\
&\leq C(M)\eta^{-3}\eps^{-\frac 13} \|\uu\|_{L^3(\T^d)}^2.
\end{aligned} \end{equation}
Next, since $\supp \,\phi_{\eta,\varepsilon} \subset \{ \varrho^\varepsilon \geq \eta\}$, $0 \leq \phi_{\eta,\varepsilon} \leq 1$, by Lemma \ref{phi}, together with Lemma \ref{keylemma-Td}, we can estimate
\begin{equation} \label{em-phi} \begin{aligned}
&\left\|  \phi_{\eta,\varepsilon} \nabla \left( \frac{|(\vr\uu)^\eps|^2}{(\vre)^2} \right) \right\|_{L^{\frac{3}{2}}(\mathbb{T}^d)}
\leq 2\left\|  \phi_{\eta,\varepsilon}  \frac{(\vr\uu)^\eps \nabla (\vr\uu)^\eps }{(\vre)^2} \right\|_{L^{\frac{3}{2}}(\mathbb{T}^d)} + 2\left\|  \phi_{\eta,\varepsilon} \frac{|(\vr\uu)^\eps|^2 \nabla \varrho^\varepsilon}{(\vre)^3}  \right\|_{L^{\frac{3}{2}}(\mathbb{T}^d)} \\
&\leq C\eta^{-2}\|(\vr\uu)^\eps\|_{L^3(\T^d)}\|\nabla(\vr\uu)^\eps\|_{L^3(\T^d)} + C\eta^{-3}\|\nabla \vre\|_{L^\infty(\T^d)}\|(\vr\uu)^\eps\|_{L^3(\T^d)}^2\\
&\leq C(M)\eta^{-2}\eps^{-\frac 23}\|\uu\|_{L^{3}(\T^d)}(\|\uu\|_{\V_{\delta}^{\frac{1}{3},3}(\T^d)} + \| \uu \|_{L^3(\T^d)}) + C\eta^{-3}\eps^{-\frac 13}\|\vr\|_{\V_\delta^{\frac 23,\infty}(\T^d)}\|\uu\|_{L^3(\T^d)}^2.
\end{aligned} \end{equation}
Combining \eqref{hu1}--\eqref{em-phi} leads to
\begin{equation} \label{hu2}
 \left\|\nabla \left(  \phi_{\eta,\varepsilon} \frac{|(\vr\uu)^\eps|^2}{(\vre)^2} \right) \right\|_{L^{\frac{3}{2}}(\mathbb{T}^d)} \leq C(M)\varepsilon^{-\frac{2}{3}}\eta^{-3} \|\uu\|_{L^{3}(\T^d)}\left(\|\uu\|_{\V_{\delta}^{\frac{1}{3},3}(\T^d)} +\|\uu\|_{L^{3}(\T^d)}\right).
\end{equation}
Inserting \eqref{hu01} and \eqref{hu2} into \eqref{hu0} yields
\begin{equation} \label{est-A2}
|(A2)| \leq C(M)\eta^{-3} \|\uu\|_{L^3}  \|\uu\|_{L^3(0,T;\V_{\delta}^{\frac{1}{3},3}(\T^d))} \left(\|\uu\|_{L^3(0,T;\V_{\delta}^{\frac{1}{3},3}(\T^d))} + \|\uu\|_{L^3}\right). 
\end{equation}
Therefore, by assumption \eqref{Tu}, 
\begin{equation}\label{h9}
\limsup_{\eta;\delta;\eps;\tau \to 0}|(A2)| = 0.
\end{equation}
We remark that, though there is a factor $\eta^{-3}$ in \eqref{est-A2}, the limit in \eqref{h9} is already $0$ after taking the limit $\delta\to 0$.
\subsection{Estimate of $(B)$}
By integration by parts we have
\begin{equation}\label{h10}
\begin{aligned}
	(B)&= -\int_{\tau}^t \int_{\T^d} [(\vr\uu\otimes\uu)^\eps - (\vr\uu)^\eps\otimes\ue] \di \left( \phi_{\eta,\varepsilon}\frac{(\vr\uu)^\eps}{\vre}\right)dxds \\
	&\quad- \int_{\tau}^t \int_{\T^d} (\vr\uu)^\eps\otimes\ue \di \left( \phi_{\eta,\varepsilon}\frac{(\vr\uu)^\eps}{\vre} \right)dxds\\
	&=: (B1) + (B2).
\end{aligned}
\end{equation}
It can be checked by using  Lemma \ref{keylemma-Td} (2) that
$$
\| (\vr\uu\otimes\uu)^\eps - (\vr\uu)^\eps\otimes\ue\|_{L^{\frac{3}{2}}(\mathbb{T}^d)} \leq C(M) \varepsilon^{\frac{2}{3}} \left( \|\uu\|_{\V_{\delta}^{\frac{1}{3},3}(\T^d)}+ \|\uu\|_{L^{3}(\T^d)}\right)\|\uu\|_{\V_{\delta}^{\frac{1}{3},3}(\T^d)}.
$$
By similar (even simpler) arguments in obtaining \eqref{hu2}, we can estimate
\begin{equation}
\left\|\di \left( \phi_{\eta,\varepsilon}\frac{(\vr\uu)^\eps}{\vre}\right)\right\|_{L^{3}(\mathbb{T}^d)}
\leq C(M)\eta^{-2}\eps^{-\frac 23}\left(\|\uu\|_{\V_\delta^{\frac 13,3}(\T^d)} + \|\uu\|_{L^3(\T^d)}\right).
\end{equation}
Thus, by H\"older's inequality,
\begin{equation}\label{h11}
\begin{aligned}
	|(B1)| &\leq C\int_\tau^t\|(\vr\uu\otimes\uu)^\eps -  (\vr\uu)^\eps\otimes\ue\|_{L^{\frac{3}{2}}(\T^d)}\left\|\di \left( \phi_{\eta,\varepsilon} \frac{(\vr\uu)^\eps}{\vre} \right) \right\|_{L^{3}(\mathbb{T}^d)}ds\\
	&\leq C(M)\eta^{-2}\left(\|\uu\|_{L^3}^2+\|\uu\|_{L^3(0,T;\V_{\delta}^{\frac{1}{3},3}(\T^d))}^2\right)\|\uu\|_{L^3(0,T;\V_{\delta}^{\frac{1}{3},3}(\T^d))}.
\end{aligned}
\end{equation}
Therefore, by \eqref{Tu}, 
\begin{equation}\label{h12}
\limsup_{\eta;\delta;\eps;\tau \to 0}|(B1)| = 0.
\end{equation}
We now turn to the estimate of $(B2)$. Using integration by parts we have
\begin{equation}\label{h13}
\begin{aligned}
	(B2) &= \int_{\tau}^t \int_{\T^d} \phi_{\eta,\varepsilon} \frac{(\vr\uu)^\eps}{\vre}\di[(\vr\uu)^\eps \otimes \ue]dxds= \int_{\tau}^t \int_{\T^d} \phi_{\eta,\varepsilon} \frac{(\vr\uu)^\eps}{\vre}[\na (\vr\uu)^\eps \ue + (\vr\uu)^\eps \na \ue] dxds\\
	 &= - \frac 12 \int_{\tau}^t \int_{\T^d} |(\vr\uu)^\eps|^2\di \left( \phi_{\eta,\varepsilon} \frac{\ue}{\vre} \right) dxds + \int_{\tau}^t\int_{\T^d} \phi_{\eta,\varepsilon} \frac{|(\vr\uu)^\eps|^2}{\vre}\na \cdot \ue dxds\\
	&= \int_{\tau}^t \int_{\T^d} \nabla \phi_{\eta,\varepsilon} |(\varrho \uu)^\varepsilon|^2 \frac{\uu^\varepsilon}{\varrho^\varepsilon}dxds  + \int_{\tau}^t \int_{\T^d} \phi_{\eta,\varepsilon} |(\vr\uu)^\eps|^2\left[\frac{1}{\vre} \na \cdot \ue - \frac 12 \na\cdot \frac{\ue}{\vre} \right] dxds \\
&= \int_{\tau}^t \int_{\T^d} \nabla \phi_{\eta,\varepsilon} |(\varrho \uu)^\varepsilon|^2 \frac{\uu^\varepsilon}{\varrho^\varepsilon}dxds + \frac 12\int_{\tau}^{t}\int_{\T^d} \phi_{\eta,\varepsilon} \frac{|(\vr\uu)^\eps|^2}{(\vre)^2}\na \cdot (\vre \ue)dxds \\
&=: (B21) + (B22).
\end{aligned}
\end{equation}
It is remarked that $(B22)=-(A3)$, and therefore we only need to estimate the term $(B21)$. We rewrite $(B21)$, using the divergence free condition, as 
\begin{align*}
	(B21) &= \eta^{-1}\int_\tau^t\int_{\T^d}\phi' \left(\frac\vre\eta \right)\frac{\ue\nabla\vre}{\vre}|(\vr\uu)^\eps|^2dxds  \\
	&= \eta^{-1}\int_\tau^t\int_{\T^d}\phi' \left(\frac\vre\eta \right)\frac{\nabla(\vre\ue - (\vr\uu)^\eps)}{\vre}|(\vr\uu)^\eps|^2dxds + \eta^{-1}\int_\tau^t\int_{\T^d}\phi' \left(\frac\vre\eta \right)\frac{\partial_t\vre}{\vre}|(\vr\uu)^\eps|^2dxds\\
	&= -\eta^{-1}\int_\tau^t\int_{\T^d}(\vre\ue - (\vr\uu)^\eps)\na\cdot\left(\phi' \left(\frac\vre\eta \right)\frac{|(\vr\uu)^\eps|^2}{\vre}\right)dxds\\
	&\quad + \eta^{-1}\int_\tau^t\int_{\T^d}\phi'\left(\frac\vre\eta \right)\frac{\partial_t\vre}{\vre}|(\vr\uu)^\eps|^2dxds.
\end{align*}	
It then follows from H\"older's inequality that
\begin{align*}	
	|(B21)| &\leq \eta^{-1}\int_\tau^t\|\vre\ue - (\vr\uu)^\eps\|_{L^3(\T^d)}\left\|\nabla\cdot\left(\phi' \left(\frac\vre\eta \right)\frac{|(\vr\uu)^\eps|^2}{\vre}\right)\right\|_{L^{\frac 32}(\T^d)}\\
	&\quad + C\eta^{-2}\|\partial_t\vre\mathbf{1}_{\eta\leq\vre\leq 2\eta}\|_{L^3}\|(\vr\uu)^\eps\mathbf{1}_{\eta\leq\vre\leq 2\eta}\|_{L^3}^2\\
	&=: (B211) + (B212).
\end{align*}
The term $(B211)$ can be estimated similarly to $(A2)$ in \eqref{hu0},
\begin{equation*}
	|(B211)| \leq C(M)\eta^{-3}\|\uu\|_{L^3} \|\uu\|_{L^3(0,T;\V_{\delta}^{\frac{1}{3},3}(\T^d))} \left(\|\uu\|_{L^3(0,T;\V_{\delta}^{\frac{1}{3},3}(\T^d))} + \|\uu\|_{L^3}\right).
\end{equation*}
Therefore, thanks to assumption \eqref{Tu},
\begin{equation*}
\limsup_{\eta;\delta;\eps;\tau \to 0}|(B211)| = 0.
\end{equation*}
The term $(B212)$ can be estimated similarly to $(A12)$, by using \eqref{A12-3} and \eqref{A12-4}, 
$$ |(B212)| \leq C  M\| \uu \mathbf{1}_{  \{\frac{1}{3}\eta \leq \varrho \leq  \frac{8}{3}\eta \} } \|_{L^3(\T^d)}^2.
$$
Hence, by the dominated convergence theorem, we obtain
\begin{equation*}
\limsup_{\eta;\delta;\eps;\tau \to 0}|(B212)| = 0.
\end{equation*}
\begin{remark}\label{divergence_free}
	We remark that we only use the free divergence free condition when estimating $(B21)$, and all the terms related to the pressure. This will be useful in the proof of Theorem \ref{toruscom} for compressible Euler equations.
\end{remark}
\subsection{Estimate of $(C)$}\footnote{When $\vr\equiv {\normalfont const}$, the term $(C)$ vanishes, and therefore the assumption $P\in L^{\frac 32}(Q_T)$ is not needed.}
Using the divergence free condition\footnote{We observe that except for the term $(B21)$, the divergence free is only used in the proof of Theorem \ref{torus} to treat the term involving the pressure. This observation will become helpful when treating compressible Euler equation in the next sections.}, we estimate
\begin{equation}\label{h2}
\begin{aligned}
(C) &= \int_{\tau}^t \int_{\T^d} \frac{\phi_{\eta,\varepsilon}}{\vre}((\vr\uu)^\eps - \vre \ue)\na P^\eps dxds - \int_{\tau}^t \int_{\T^d} \nabla \phi_{\eta,\varepsilon} \, \uu^\varepsilon P^\varepsilon dxds \\
&=: (C1) + (C2).
\end{aligned}
\end{equation}
From Lemma \ref{phi} (i), the fact $0 \leq \phi_{\eta,\varepsilon} \leq 1$, H\"older's inequality and Lemma \ref{keylemma-Td} (2), we get
\begin{equation}\label{h2_1}
\begin{aligned}
|(C1)| &\leq C\eta^{-1} \int_0^T\|(\vr\uu)^\eps - \vre\ue\|_{L^{3}(\T^d)}\|\na P^\eps\|_{L^{\frac{3}{2}}(\T^d)}ds\\
&\leq C\eta^{-1} \int_0^T\eps^{\frac{2}{3}}\|\vr\|_{\V_\delta^{\frac{2}{3},\infty}(\T^d)}\eps^{\frac{1}{3}}\|\uu\|_{\V_\delta^{\frac{1}{3},3}(\T^d)}\eps^{-1}\|P\|_{L^{\frac{3}{2}}(\T^d)}ds\\
&\leq C\eta^{-1} \|\vr\|_{L^\infty(0,T;\V_\delta^{\frac{2}{3},\infty}(\T^d))}\|\uu\|_{L^3(0,T;\V_{\delta}^{\frac{1}{3},3}(\T^d))}\|P\|_{L^{\frac{3}{2}}}.
\end{aligned}
\end{equation}
By assumption \eqref{Tu}, we have
\begin{equation*}
\limsup_{\eta;\delta;\eps;\tau \to 0}|(C1)| = 0.
\end{equation*}

Next we estimate the term $(C2)$. By integration by parts,
\begin{align*}
(C2) &= \eta^{-1}\int_\tau^t \int_{\T^d} \phi' \left( \frac{\vre}{\eta} \right) \nabla \vre \ue P^\varepsilon dxds \\
&= \eta^{-1}\int_\tau^t\int_{\T^d}\phi' \left( \frac{\vre}{\eta} \right) {\na(\vre\ue - (\vr\uu)^\eps)}P^\eps dxds + \eta^{-1} \int_\tau^t\int_{\T^d}\phi' \left( \frac{\vre}{\eta} \right){\partial_t\vre}P^\eps dxds\\
&\leq -\eta^{-1}\int_\tau^t\int_{\T^d}(\vre\ue - (\vr\uu)^\eps)\na\cdot\left(\phi'\left(\frac{\vre}{\eta}\right) P^\eps \right)dxds \\
&\quad + 2\int_\tau^t\int_{\T^d}\left|\phi' \left( \frac{\vre}{\eta} \right) \frac{\partial_t\vre}{\vre}P^\eps \mathbf{1}_{ \{ \eta\leq\vre\leq2\eta \}}\right|dxds.
\end{align*}
Hence, using Holder inequality, we deduce
\begin{align*}
|(C2)|&\leq \eta^{-2}\int_\tau^t\|\vre\ue - (\vr\uu)^\eps\|_{L^3(\T^d)}\left\|\phi''\left(\frac{\vre}{\eta}\right)\na\vre P^\eps  \right\|_{L^{\frac 32}(\T^d)}ds\\
&\quad + \eta^{-1}\int_\tau^t\|\vre\ue - (\vr\uu)^\eps\|_{L^3(\T^d)}\left\|\phi'\left(\frac{\vre}{\eta}\right)\na P^\eps  \right\|_{L^{\frac 32}(\T^d)}ds\\
&\quad + C\left\|(\vre)^{-1} \mathbf{1}_{ \{ \eta\leq\vre\leq2\eta \}} \right\|_{L^{r_1}}\|\partial_t\vre \mathbf{1}_{ \{ \eta\leq\vre\leq2\eta \}}\|_{L^{r_2}} \|P^\eps \mathbf{1}_{ \{ \eta\leq\vre\leq2\eta \}}\|_{L^{\frac{3}{2}}} \\
&=: (C21) + (C22) + (C23),
\end{align*}
where $r_1, r_2$ are in \eqref{Trho2}. Since $\phi''$ is bounded, due to Lemma \ref{keylemma-Td} (2), we get
\begin{align*}
	|(C21)| &\leq C\eta^{-2}\int_\tau^t\eps^{\frac 23}\|\vr\|_{\V_\delta^{\frac 23,\infty}(\T^d)}\eps^{\frac 13}\|\uu\|_{\V_\delta^{\frac 13,3}(\T^d)}\eps^{-\frac 13}\|\vr\|_{\V_{\delta}^{\frac 23,\infty}(\T^d)}\|P\|_{L^{\frac 32}(\T^d)}dt\\
	&\leq C\eta^{-2}\eps^{\frac 23}\|\vr\|_{L^\infty(0,T;\V_\delta^{\frac 23,\infty}(\T^d))}^2\|\uu\|_{L^3(0,T;\V_\delta^{\frac 13,3}(\T^d))}\|P\|_{L^{\frac 32}}.
\end{align*}
Thus
\begin{equation*}
\limsup_{\eta;\delta;\eps;\tau \to 0}|(C21)| = 0.
\end{equation*}
The term $(C22)$ can be estimated similarly to $(C1)$ in \eqref{h2_1} using the fact that $\phi'$ is uniformly bounded \eqref{def_phi}, and thus we get
$$
\limsup_{\eta;\delta;\eps;\tau \to 0}|(C22)| = 0.
$$
To estimate $(C23)$, we observe, due to Lemma \ref{phi}, that
\begin{equation} \label{C23-1} \left\|(\vre)^{-1} \mathbf{1}_{ \{ \eta\leq\vre\leq2\eta \}} \right\|_{L^{r_1}} \leq C\left\| \varrho^{-1} \mathbf{1}_{ \{ \frac{1}{3}\eta\leq \varrho \leq \frac{8}{3}\eta \}} \right\|_{L^{r_1}}
\leq C\left\| \varrho^{-1} \mathbf{1}_{ \{ \varrho \leq \delta_0 \}} \right\|_{L^{r_1}} \leq CM
\end{equation}
and (by using the argument leading to \eqref{A12-3} and replacing $3$ by $r_2$)
\begin{equation} \label{C23-2} \|\partial_t\vre \mathbf{1}_{ \{ \eta\leq\vre\leq2\eta \}}\|_{L^{r_2}} \leq C M.
\end{equation}
Similarly
\begin{equation} \label{C23-3} \|P^\eps \mathbf{1}_{ \{ \eta\leq\vre\leq2\eta \}}\|_{L^{\frac{3}{2}}} \leq C \|P \mathbf{1}_{ \{ \frac{1}{3}\eta\leq\varrho\leq \frac{8}{3}\eta \}}\|_{L^{\frac{3}{2}}}.
\end{equation}
From the above estimates, we obtain
$$ (C23) \leq C(M) \|P \mathbf{1}_{ \{ \frac{1}{3}\eta\leq\varrho\leq \frac{8}{3}\eta \}}\|_{L^{\frac{3}{2}}}.
$$
By dominated convergence theorem, 
$$
\limsup_{\eta;\delta;\eps;\tau \to 0}|(C23)| = 0.
$$
Collecting the above estimates, we eventually deduce that
$$
\limsup_{\eta;\delta;\eps;\tau \to 0}|(C)| = 0.
$$

\begin{remark}\label{P} The condition $\|\vr^{-1}\|_{L^{r_1}} < +\infty$ is only needed to control $(C23)$ resulting from the pressure term. In the compressible equation, this condition is not needed.
Removing this technical condition in the incompressible case remains as an interesting open problem.
\end{remark}

\subsection{Conclusion}\label{conclusion}
From the previous estimates we have
\begin{equation*}
\limsup_{\eta;\delta;\eps;\tau \to 0}	\left|\int_\tau^t\int_{\T^d}\partial_t\left[\phi_{\eta,\varepsilon} \frac{|(\vr\uu)^\eps|^2}{\vre}\right](x,s)dxds\right| =0.
\end{equation*}
By \eqref{ini1} and \eqref{ini2},  $\vr(\cdot,\tau) \rightharpoonup \vr_0$ and $(\vr\uu)(\cdot,\tau) \rightharpoonup \vr_0\uu_0$ in ${\mathcal D}'(\Omega)$, so for every $x\in \mathbb{T}^d$, 
$$
	\vre(x,\tau)\to\varrho_0^\varepsilon(x), \quad
(\vr\uu)^\eps(x,\tau)\to (\vr_0\uu_0)^\eps(x),$$
	as $\tau \to 0$. Thanks to dominated convergence theorem, it yields 
\begin{equation*}
\limsup_{\eta\to 0}\limsup_{\eps\to 0}\left| \int_{\T^d}\left( \phi_{\eta,\varepsilon} \frac{|(\vr\uu)^\eps|^2}{\vre}\right)(x,t)dx-	\int_{\T^d}\left( \phi_{\eta,\varepsilon} \frac{|(\vr_0\uu_0)^\eps|^2}{\vre_0}\right)(x)dx\right| =0.
\end{equation*}
Thus, by the standard property of convolution, we derive 
\begin{equation*}
\limsup_{\eta\to 0}\limsup_{\eps\to 0}\left| \int_{\T^d}\left( \phi_{\eta,\varepsilon} \varrho |\uu|^2 \right)(x,t)dx-	\int_{\T^d}\left( \phi_{\eta,\varepsilon} \varrho_0 |\uu_0|^2\right)(x)dx\right| =0.
\end{equation*}
This implies \eqref{conservation1} by using the dominated convergence theorem. The proof is complete. \qeda
\section{Proof of Theorem \ref{bounded}}\label{sec:bounded}


The proof of Theorem \ref{bounded} is similar to that of Theorem \ref{torus}, except that we have to take care of the boundary layers when taking integration by parts. Recalling the smooth version of \eqref{E} as
\begin{equation}\label{Brho_approx}
	\pa_t\vre + \di (\vr\uu)^\eps = 0 \quad \text{ in } \quad \Omega_{2\eps},
\end{equation}
and
\begin{equation}\label{Bu_approx}
	\pa_t(\vr\uu)^\eps + \di (\vr\uu\otimes\uu)^\eps + \na P^\eps = 0 \quad \text{ in } \quad \Omega_{2\eps}.
\end{equation}
Take $0 < \eps < \eps_1/10 < \eps_2/10<\delta_0/100$. Note that if $\eps_2 \to 0$ then $\eps, \eps_1 \to 0$. Multiplying \eqref{Bu_approx} by $\frac{\phi_{\eta,\eps}}{\vre}(\vr\uu)^{\eps}$ then integarting on $(\tau,t)\times\Omega_{\eps_2}$, with $0<\tau<t<T$, yield
\begin{equation}\label{k1}
\begin{aligned}
	\int_{\tau}^t \int_{\Omega_{\varepsilon_2}} \frac{\phi_{\eta,\eps}}{\vre}(\vr\uu)^\eps\pa_t(\vr\uu)^\eps dxds &+ \int_{\tau}^t \int_{\Omega_{\varepsilon_2}} \frac{\phi_{\eta,\eps}}{\vre}(\vr\uu)^\eps \di(\vr\uu\otimes\uu)^\eps dxds\\
	&+ \int_{\tau}^t \int_{\Omega_{\varepsilon_2}} \frac{\phi_{\eta,\eps}}{\vre}(\vr\uu)^\eps \na P^\eps dxds = 0.
\end{aligned}
\end{equation} 
For $\eps_3 > 0$ small, we integrate \eqref{k1} with respect to $\eps_2$ on $(\eps_1, \eps_1+\eps_3)$ to get
\begin{equation}\label{k2}
\begin{aligned}
	\frac{1}{\varepsilon_3} \int_{\varepsilon_1}^{\varepsilon_1 + \varepsilon_3}  \int_{\tau}^t \int_{\Omega_{\varepsilon_2}} \frac{\phi_{\eta,\eps}}{\vre}(\vr\uu)^\eps\pa_t(\vr\uu)^\eps \dxte &+ \frac{1}{\varepsilon_3} \int_{\varepsilon_1}^{\varepsilon_1 + \varepsilon_3}  \int_{\tau}^t \int_{\Omega_{\varepsilon_2}} \frac{\phi_{\eta,\eps}}{\vre}(\vr\uu)^\eps \di(\vr\uu\otimes\uu)^\eps \dxte\\
	&+ \frac{1}{\varepsilon_3} \int_{\varepsilon_1}^{\varepsilon_1 + \varepsilon_3}  \int_{\tau}^t \int_{\Omega_{\varepsilon_2}} \frac{\phi_{\eta,\eps}}{\vre}(\vr\uu)^\eps \na P^\eps \dxte = 0.
\end{aligned}
\end{equation}
Denote by $(D)$, $(E)$ and $(F)$ the three terms on the left-hand side of \eqref{k2}, which will be estimated separately in the following subsections. Let $M_{\eps_1}$ be a constant such that 
\begin{align*}
\|\varrho\|_{L^\infty}+\|\varrho\|_{L^{\infty}(0,T;\V_{\varepsilon_1/4}^{\frac{2}{3},\infty}(\Omega_{\eps_1/2}))}+\|\varrho^{-1}\mathbf{1}_{\{\vr \leq \delta_0\}}\|_{L^{r_1}}+\|\partial_t\vr \mathbf{1}_{\{\vr\leq \delta_0\}}\|_{L^{r_2}}\leq M_{\eps_1}. 
\end{align*}
It's worth mentioning that in the following, we let successively $\tau \to 0$, $\eps \to 0$, $\eps_1 \to 0$, $\eps_3 \to 0$ and then $\eta \to 0$. Therefore, at some estimates, after letting $\eps \to 0$, constants depending on $M_{\eps_1}$ (usually denoted by $C(M_{\eps_1})$) vanishes for each $\eps_1>0$, therefore we  will be able to send  $\eps_1 \to 0$ afterwards without encountering any issue.
\subsection{Estimate of $(D)$}
The term $(D)$ is rewritten as
\begin{equation}\label{k7}
\begin{aligned}
(D) &= \frac 12 \frac{1}{\varepsilon_3}\int_{\varepsilon_1}^{\varepsilon_1 + \varepsilon_3}  \int_{\tau}^t \int_{\Omega_{\varepsilon_2}}\phi_{\eta,\eps} \pa_t\left(\frac{|(\vr\uu)^\eps|^2}{\vre}\right)dxdsd\varepsilon_2\\
&\quad - \frac 12 \frac{1}{\varepsilon_3}\int_{\varepsilon_1}^{\varepsilon_1 + \varepsilon_3}  \int_{\tau}^t \int_{\Omega_{\varepsilon_2}} \frac{\phi_{\eta,\eps}}{(\vre)^2}\di[(\vr\uu)^\eps - \vre\ue]|(\vr\uu)^\eps|^2dxdsd\varepsilon_2\\
&\quad - \frac 12 \frac{1}{\varepsilon_3}\int_{\varepsilon_1}^{\varepsilon_1 + \varepsilon_3}  \int_{\tau}^t \int_{\Omega_{\varepsilon_2}} \frac{1}{(\vre)^2}\di(\vre\ue)|(\vr\uu)^\eps|^2dxdsd\varepsilon_2\\
&=: (D1) + (D2) + (D3).
\end{aligned}
\end{equation}
For $(D1)$ we have
\begin{align*}
(D1) &= \frac 12\frac{1}{\eps_3}\int_\tau^t\int_{\eps_1}^{\eps_1+\eps_3}\int_{\Omega_{\eps_2}}\partial_t\left[\phi_{\eta,\eps}\frac{|(\vr\uu)^\eps|^2}{\vre}(x,s)\right]\dxte\\
&\quad  - \frac 12 \Binttt \partial_t \phi_{\eta,\eps}\frac{|(\vr\uu)^\eps|^2}{\vre}\dxte\\
&=: (D11) + (D12).
\end{align*}
Since $(D11)$ is the desired term, we only estimate $(D12)$. Using similar arguments to estimate $(A12)$, we have
\begin{align*}
	|(D12)| &\leq C\eta^{-2}\int_{\tau}^t\int_{\Omega}|\partial_t\vre||(\vr\uu)^\eps|^2\mathbf{1}_{\{\eta\leq \vre\leq 2\eta\}}dxds\\
	&\leq C\eta^{-2}\|\partial_t\vre\mathbf{1}_{\{\eta\leq \vre\leq 2\eta\}}\|_{L^3}\|(\vr\uu)^\eps \mathbf{1}_{\{\eta\leq \vre\leq 2\eta\}}\|_{L^3}^2\\
	&\leq C\|\partial_t \vr \mathbf{1}_{\{\frac 13\eta\leq \vr\leq \frac 83\eta\}}\|_{L^3}\|\uu \mathbf{1}_{\{\frac 13\eta\leq \vr\leq \frac 83\eta\}}\|_{L^3}^2,
\end{align*}
thus
\begin{equation*}
 \limsup_{\eta;\eps_3;\eps_1;\eps;\tau \to 0}|(D12)| = 0.
\end{equation*}
It remains to estimate $(D2)$ since $(D3)$ will be estimated together with $(E3)$ later (see \eqref{k15}). Using integration by parts we have
\begin{equation*}
\begin{aligned}
	|(D2)| &\leq \frac 12 \left|\frac{1}{\varepsilon_3}\int_{\varepsilon_1}^{\varepsilon_1 + \varepsilon_3}  \int_{\tau}^t \int_{\partial \Omega_{\varepsilon_2}}\frac{\phi_{\eta,\eps}}{(\vre)^2} |(\vr\uu)^\eps|^2  [(\vr\uu)^\eps - \vre\ue]\cdot n(\theta) d{\mathcal H}^{d-1}(\theta)ds d\varepsilon_2 \right|\\
	&\quad + \frac 12 \left|\frac{1}{\varepsilon_3}\int_{\varepsilon_1}^{\varepsilon_1 + \varepsilon_3}  \int_{\tau}^t \int_{\Omega_{\varepsilon_2}} [(\vr\uu)^\eps - \vre\ue]\na\left(\phi_{\eta,\eps}\frac{|(\vr\uu)^\eps|^2}{(\vre)^2}\right)dxdsd\varepsilon_2 \right|\\
	&=: (D21) + (D22).
\end{aligned}
\end{equation*}
The term $(D22)$ can be estimated similarly to $(A2)$ as
\begin{equation}\label{k8}
\begin{aligned}
	(D22) &\leq \frac 12 \int_\tau^t\int_{\Omega_{\eps_1}}\left|[(\vr\uu)^\eps - \vre\ue]\na\left(\phi_{\eta,\eps}\frac{|(\vr\uu)^\eps|^2}{(\vre)^2}\right)\right| dxds\\
	&\leq C(M_{\eps_1})\eta^{-3}  \| \uu \|_{L^3(0,T;{\mathcal V}_{\varepsilon}^{\frac{1}{3},3}(\Omega_{\frac{\eps_1}{2}}) ) } \left[\|\uu\|_{L^3}^2+ \|\uu\|_{L^3}\|\uu\|_{L^3(0,T;{\mathcal V}_{\varepsilon}^{\frac{1}{3},3}(\Omega_{\frac{\eps_1}{2}}) ) }\right].
\end{aligned}
\end{equation} Therefore, by using \eqref{Tu'},
\begin{equation}\label{k9}
 \limsup_{\eta;\eps_3;\eps_1;\eps;\tau \to 0}(D22) = 0.
\end{equation}
For $(D21)$ we use the coarea formula:  for any $0<r_1<r_2< 1$, and $g\in L^1(\Omega_{r_1}\backslash\Omega_{r_2})$, 
\begin{align}\label{coareaformula} \int_{\Omega_{r_1}\backslash\Omega_{r_2}}g(x)dx=\int_{r_1}^{r_2}\int_{\partial\Omega_{\nu}}g(\theta)d\mathcal{H}^{d-1}(\theta) d\nu,
\end{align}
 and $\eps_3 \approx \mathcal{L}^d(\Omega_{\eps_1}\backslash\Omega_{\eps_1+\eps_3})$, and $\phi_{\eta,\eps} = 0$ if $\vre\leq \eta$, to get
\begin{align*}
(D21)&\leq C\left|\int_\tau^t\fint_{\Omega_{\eps_1}\backslash\Omega_{\eps_1+\eps_3}}\frac{\phi_{\eta,\eps}}{(\vre)^2} |(\vr\uu)^\eps|^2 [(\vr\uu)^\eps - \vre\ue]\cdot n(x)dxds \right|\\
&\leq C\eta^{-2}\left|\int_\tau^t\fint_{\Omega_{\eps_1}\backslash\Omega_{\eps_1+\eps_3}}|(\vr\uu)^\eps|^2 [(\vr\uu)^\eps - \vre\ue]\cdot n(x)dxds \right|.
\end{align*}
Since $\varrho \in L^\infty(Q_T)$ and $\uu\in L^3(Q_T)$, we have
\begin{equation}\label{k10}
 \limsup_{\eta;\eps_3;\eps_1;\eps;\tau \to 0} |(D21)| = 0.
\end{equation}

\subsection{Estimate of $(E)$}
Using integration by parts we have
\begin{equation}\label{k11}
\begin{aligned}
	(E) &= \frac{1}{\varepsilon_3}\int_{\varepsilon_1}^{\varepsilon_1 + \varepsilon_3}  \int_{\tau}^t \int_{\partial \Omega_{\varepsilon_2}} \frac{\phi_{\eta,\eps}}{\vre}(\vr\uu)^\eps (\vr\uu\otimes \uu)^\eps n(\theta)d{\mathcal H}^{d-1}(\theta)ds d\varepsilon_2\\
	&\quad - \frac{1}{\varepsilon_3}\int_{\varepsilon_1}^{\varepsilon_1 + \varepsilon_3}  \int_{\tau}^t \int_{\Omega_{\varepsilon_2}} (\vr\uu\otimes\uu)^\eps \di\left(\phi_{\eta,\eps} \frac{(\vr\uu)^\eps}{\vre}\right) dxdsd\varepsilon_2\\
	&=: (E1) - \frac{1}{\varepsilon_3}\int_{\varepsilon_1}^{\varepsilon_1 + \varepsilon_3}  \int_{\tau}^t \int_{\Omega_{\varepsilon_2}} [(\vr\uu\otimes\uu)^\eps - (\vr\uu)^\eps\otimes\ue] \di\left(\phi_{\eta,\eps} \frac{(\vr\uu)^\eps}{\vre}\right) dxdsd\varepsilon_2\\
	&\quad - \frac{1}{\varepsilon_3}\int_{\varepsilon_1}^{\varepsilon_1 + \varepsilon_3}  \int_{\tau}^t \int_{\Omega_{\varepsilon_2}} (\vr\uu)^\eps\otimes\ue \di\left(\phi_{\eta,\eps} \frac{(\vr\uu)^\eps}{\vre}\right) dxdsd\varepsilon_2\\
	&=: (E1) + (E2) + (E3).
\end{aligned}
\end{equation}
By the coarea formula \eqref{coareaformula}, we have
\begin{align*}
(E1)=\frac{1}{\varepsilon_3}  \int_{\tau}^t \int_{\Omega_{\varepsilon_1} \setminus \Omega_{\varepsilon_1 + \varepsilon_3 } } \frac{\phi_{\eta,\eps}}{\vre}(\vr\uu)^\eps (\vr\uu\otimes \uu)^\eps n(x)dxds.
\end{align*}
Therefore, letting successively $\tau \to 0$, $\varepsilon \to 0$ and $\varepsilon_1 \to 0$, then using the fact $\eps_3 \approx \mathcal L^d(\Omega\backslash \Omega_{\eps_3})$, $\phi_{\eta,\eps} = 0$ for $\vre\leq \eta$, and H\"older's inequality, we obtain
\begin{align*} 
\limsup_{\eps_1;\eps;\tau\to 0}|(E1)| &\leq \frac{C(\|\vr\|_{L^\infty})}{\varepsilon_3}\eta^{-1}  \int_{0}^T \int_{\Omega \setminus \Omega_{\varepsilon_3 } }  |\uu|^2  |\uu \cdot n|dxds \\
&\leq C(\|\vr\|_{L^\infty})\eta^{-1} \left( \int_0^T \fint_{\Omega \setminus \Omega_{\varepsilon_3}}   |\uu|^{3} dxds \right)^{\frac{2}{3}} \left( \int_0^T \fint_{\Omega \setminus \Omega_{\varepsilon_3}} |\uu \cdot n|^3 dxds \right)^{\frac{1}{3}}.
\end{align*}
Thus, by assumption \eqref{uboundary}, we derive
\begin{equation*}
	 \limsup_{\eta;\eps_3;\eps_1;\eps;\tau \to 0}|(E1)| = 0.
\end{equation*}
By proceeding as in estimating $(B1)$ we derive 
\begin{equation}\label{k13}
	\begin{aligned}
	|(E2)| &\leq C\int_{\tau}^t \int_{\Omega_{\varepsilon_1}} |(\vr\uu\otimes\uu)^\eps - (\vr\uu)^\eps\otimes\ue| \left| \nabla\cdot\left(\phi_{\eta,\eps} \frac{(\vr\uu)^\eps}{\vre}\right) \right| dxds  \\
	&\leq C\int_0^T\|(\vr\uu\otimes\uu)^\eps - (\vr\uu)^\eps\otimes \ue\|_{L^{\frac{3}{2}}(\Omega_{\varepsilon_1})}\left\|\di\left(\phi_{\eta,\eps} \frac{(\vr\uu)^\eps}{\vre}\right)\right\|_{L^3(\Omega_{\varepsilon_1})}ds\\
	&\leq C(M_{\eps_1})\eta^{-2}\left(\|\uu\|_{L^3}^2 + \|\uu\|_{L^3(0,T;{\mathcal V}_{\varepsilon}^{ \frac{1}{3},3}(\Omega_{\frac{\eps_1}{2}})) }^2\right)\|\uu\|_{L^3(0,T;{\mathcal V}_{\varepsilon}^{ \frac{1}{3},3}(\Omega_{\frac{\eps_1}{2}})) }.
	\end{aligned}
\end{equation}
Therefore, by assumption \eqref{Tu'}, it follows that
\begin{equation}\label{k14}
\limsup_{\eta;\eps_3;\eps_1;\eps;\tau \to 0}|(E2)| = 0.
\end{equation}
\subsection{Estimate of $(F)$}
By using the divergence free condition
we can split $(F)$ as
\begin{equation}\label{k3}
\begin{aligned}
(F) &= \frac{1}{\varepsilon_3} \int_{\varepsilon_1}^{\varepsilon_1 + \varepsilon_3}  \int_{\tau}^t \int_{\Omega_{\varepsilon_2}} \frac{\phi_{\eta,\eps}}{\vre}((\vr\uu)^\eps - \vre\ue)\na P^\eps dx ds d \varepsilon_2 + \frac{1}{\varepsilon_3} \int_{\varepsilon_1}^{\varepsilon_1 + \varepsilon_3}  \int_{\tau}^t \int_{\Omega_{\varepsilon_2}} \phi_{\eta,\eps}\ue\na P^\eps dxdsd\varepsilon_2\\
&=: (F1) + \frac{1}{\varepsilon_3} \int_{\varepsilon_1}^{\varepsilon_1 + \varepsilon_3}  \int_{\tau}^t \int_{\partial \Omega_{\varepsilon_2}} \phi_{\eta,\eps}P^\eps \uu^\varepsilon\cdot n(\theta) d{\mathcal H}^{d-1}(\theta) ds d\varepsilon_2\\
&\quad - \Binttt \na\phi_{\eta,\eps} \ue P^\eps \dxte\\
&=: (F1) + (F2) + (F3).
\end{aligned}
\end{equation}
The term $(F1)$ is estimated similarly to the term $(C1)$ in the case $\Omega=\T^d$, thus by Lemma \ref{keylemma-boundeddomain} (2), we obtain
\begin{equation}\label{k4}
\begin{aligned}
|(F1)| &\leq  \int_0^T\int_{\Omega_{\eps_1}}\left|\frac{\phi_{\eta,\eps}}{\vre}\right| |(\vr\uu)^\eps - \vre\ue| |\na P^\eps| dxds\\
&\leq C\eta^{-1}\|\varrho\|_{L^\infty(0,T;{\mathcal V}_{\frac{\varepsilon_1}{4}}^{\frac{2}{3},\infty}(\Omega_{\frac{\eps_1}{2}}))}\|\uu\|_{L^3(0,T;{\mathcal V}_{\varepsilon}^{ \frac{1}{3},3}(\Omega_{\frac{\eps_1}{2}})) }\|P\|_{L^{\frac{3}{2}}}.
\end{aligned}
\end{equation}
Hence, by assumption \eqref{Tu'},
\begin{equation}\label{k5}
\limsup_{\eta;\eps_3;\eps_1;\eps;\tau \to 0}|(F1)| = 0.
\end{equation}
For the term $(F2)$ we use the coarea formula \eqref{coareaformula} to obtain
$$
(F2) = \frac{1}{\eps_3}\int_\tau^t\int_{\Omega_{\eps_1}\backslash \Omega_{\eps_1+\eps_3}}\phi_{\eta,\eps}P^\eps \uu^\varepsilon \cdot n(x)dxds.
$$
By using the fact $\eps_3 \approx \mathcal L^d(\Omega\backslash \Omega_{\eps_3})$ and H\"older's inequality, we obtain
\begin{equation*}
\begin{aligned}
\limsup_{\varepsilon_1;\varepsilon;\tau \to 0}|(F2)| &= \left|\frac{1}{\eps_3}\int_{\tau}^t\int_{\Omega \backslash \Omega_{\eps_3}}\phi_{\eta} P(x,s) \uu(x,s) \cdot n(x)dxds\right|\\
&\leq C\left(\int_0^T\fint_{\Omega\backslash\Omega_{\eps_3}}|P(x,s)|^{\frac{3}{2}}dxds\right)^{\frac{2}{3}}\left(\int_0^T\fint_{\Omega\backslash\Omega_{\eps_3}}|\uu(x,s)\cdot n(x)|^{3}dxds\right)^{\frac{1}{3}}.
\end{aligned}
\end{equation*}
By the assumption \eqref{Pboundary} we get
\begin{equation}\label{k6}
\limsup_{\eta;\eps_3;\eps_1;\eps;\tau \to 0}|(F2)|=0.
\end{equation}
It remains to estimate $(F3)$. We use similar arguments in the estimation of $(C2)$. More precisely,
\begin{align*}
	|(F3)| &\leq \eta^{-1}\left|\Binttt \phi'\left(\frac{\vre}{\eta}\right)\na( \vre \ue - (\vr\uu)^\eps)  P^\eps \dxte\right|\\
	&\quad + \eta^{-1}\left|\Binttt \phi'\left(\frac{\vre}{\eta}\right) \partial_t \vre P^\eps\dxte\right|\\
	&\leq \eta^{-1}\left|\BintttB\phi'\left(\frac{\vre}{\eta}\right) (\vre \ue - (\vr\uu)^\eps) P^\eps \dHte  \right|\\
	&\quad + \eta^{-1}\left|\Binttt(\vre\ue-(\vr\uu)^\eps)\left[\phi''\left(\frac{\vre}{\eta}\right)\frac{\na\vre}{\eta}P^\eps + \phi'\left(\frac{\vre}{\eta}\right)\na P^\eps \right] \dxte \right|\\
	&\quad + \eta^{-1}\left|\Binttt \phi'\left(\frac{\vre}{\eta}\right) \partial_t \vre P^\eps\dxte\right|\\
	&=: (F31) + (F32) + (F33). 
\end{align*}
We estimate $(F32)$ similarly to $(C1)$ and $(C21)$ to get
\begin{align*}
	|(F32)| &\leq C\eta^{-1}\|\vr\|_{L^\infty(0,T;\V_{\eps}^{\frac 23,\infty}(\Omega_{\frac{\eps_1}{2}}))}\|\uu\|_{L^3(0,T;\V_\eps^{\frac 13,3}(\Omega_{\frac{\eps_1}{2}}))}\|P\|_{L^{\frac 32}}\\
	&\quad + C\eta^{-2}\eps^{\frac 23}\|\vr\|_{L^\infty(0,T;\V_{\eps}^{\frac 23,\infty}(\Omega_{\frac{\eps_1}{2}}))}^2\|\uu\|_{L^3(0,T;\V_\eps^{\frac 13,3}(\Omega_{\frac{\eps_1}{2}}))}\|P\|_{L^{\frac 32}},
\end{align*}
and thus$$\limsup_{\eta;\eps_3;\eps_1;\eps;\tau \to 0}|(F32)| = 0.$$
Using similar arguments in estimating $(F2)$, taking into account that $\vr\in L^\infty(Q_T)$, we have
\begin{equation*}
	\limsup_{\varepsilon_1;\varepsilon;\tau \to 0}|(F31)| \leq C\left(\int_0^T\fint_{\Omega\backslash\Omega_{\eps_3}}|P(x,s)|^{\frac{3}{2}}dxds\right)^{\frac{2}{3}}\left(\int_0^T\fint_{\Omega\backslash\Omega_{\eps_3}}|\uu(x,s)\cdot n(x)|^{3}dxds\right)^{\frac{1}{3}}
\end{equation*}
and therefore, thanks to assumption \eqref{Pboundary}
$$\limsup_{\eta;\eps_3;\eps_1;\eps;\tau \to 0}|(F31)| = 0.$$
Finally, $(F33)$ is estimated similarly to $(C23)$ as
\begin{equation*}
	|(F33)| \leq C\|\vr^{-1}\mathbf{1}_{\{\vr\leq \frac 83\eta\}}\|_{L^{r_1}}\|\partial_t\vr\mathbf{1}_{\{\vr\leq \frac 83\eta\}}\|_{L^{r_2}}\|P\mathbf{1}_{\{\vr\leq \frac 83\eta\}}\|_{L^{\frac 32}},
\end{equation*}
thus
$$\limsup_{\eta;\eps_3;\eps_1;\eps;\tau \to 0}|(F33)| = 0.$$
\subsection{Estimate of $(D3) + (E3)$}
Using integration by parts we compute (similarly to $(A3) + (B2) = 0$)
\begin{align}\label{k15}
(E3) =& -\frac{1}{\varepsilon_3}\int_{\varepsilon_1}^{\varepsilon_1 + \varepsilon_3}  \int_{\tau}^t \int_{\partial \Omega_{\varepsilon_2}}\phi_{\eta,\eps} (\vr\uu)^\eps\otimes \ue \frac{(\vr\uu)^\eps}{\vre}\cdot n(\theta)d{\mathcal H}^{d-1}(\theta)ds d\varepsilon_2\nonumber\\
&\quad + \frac{1}{\varepsilon_3}\int_{\varepsilon_1}^{\varepsilon_1 + \varepsilon_3}  \int_{\tau}^t \int_{\Omega_{\varepsilon_2}}\phi_{\eta,\eps} \frac{(\vr\uu)^\eps}{\vre}\di[(\vr\uu)^\eps\otimes \ue] dxdsd\varepsilon_2\nonumber\\
&=: (E31) + \frac{1}{\varepsilon_3}\int_{\varepsilon_1}^{\varepsilon_1 + \varepsilon_3}  \int_{\tau}^t \int_{\Omega_{\varepsilon_2}}\phi_{\eta,\eps} \frac{(\vr\uu)^\eps}{\vre}[\di (\vr\uu)^\eps \ue + (\vr\uu)^\eps \di\ue] dxdsd\varepsilon_2\nonumber\\
&= (E31) + \frac 12 \frac{1}{\varepsilon_3}\int_{\varepsilon_1}^{\varepsilon_1 + \varepsilon_3}  \int_{\tau}^t \int_{\partial \Omega_{\varepsilon_2}}\phi_{\eta,\eps} |(\vr\uu)^\eps|^2\frac{\ue}{\vre}\cdot n(\theta)d{\mathcal H}^{d-1}(\theta)ds d\varepsilon_2 \nonumber\\
&\quad - \frac 12 \frac{1}{\varepsilon_3}\int_{\varepsilon_1}^{\varepsilon_1 + \varepsilon_3}  \int_{\tau}^t \int_{\Omega_{\varepsilon_2}} |(\vr\uu)^\eps|^2\di\left(\frac{\phi_{\eta,\eps}}{\vre}\ue\right) dxdsd\varepsilon_2\nonumber\\
&\quad + \Binttt|(\vr\uu)^\eps|^2 \frac{\phi_{\eta,\eps}}{\vre}\di \ue \dxte\nonumber\\
&=: (E31) + (E32) - \frac 12 \frac{1}{\varepsilon_3}\int_{\varepsilon_1}^{\varepsilon_1 + \varepsilon_3}  \int_{\tau}^t \int_{\Omega_{\varepsilon_2}} |(\vr\uu)^\eps|^2\na \phi_{\eta,\eps}\frac{\ue}{\vre} dxdsd\varepsilon_2\nonumber\\
&\quad + \frac 12 \frac{1}{\varepsilon_3}\int_{\varepsilon_1}^{\varepsilon_1 + \varepsilon_3}  \int_{\tau}^t \int_{\Omega_{\varepsilon_2}} |(\vr\uu)^\eps|^2\frac{\di(\vre\ue)}{(\vre)^2} dxdsd\varepsilon_2\nonumber\\
&=: (E31) + (E32) - (E33) - (D3).
\end{align}
Therefore, it remains only to estimate $(E31)$, $(E32)$, and $(E33)$. By using a similar argument as in estimating $(E1)$, together with the coarea formula, H\"older's inequality and the fact $\eps_3 \approx \mathcal L^d(\Omega\backslash \Omega_{\eps_3})$, we obtain 
\begin{align*} 
\limsup_{\varepsilon_1;\varepsilon;\tau \to 0}(|(E31)|+|(E32)|) \leq C\eta^{-1}\left( \int_0^T \fint_{\Omega \setminus \Omega_{\varepsilon_3}}   |\uu|^{3} dxds \right)^{\frac{2}{3}} \left( \int_0^T \fint_{\Omega \setminus \Omega_{\varepsilon_3}} |\uu \cdot n|^3 dxds \right)^{\frac{1}{3}}.
\end{align*}
By assumption \eqref{uboundary}, we deduce
\begin{equation} \label{k16}
\limsup_{\eta;\eps_3;\eps_1;\eps;\tau \to 0}\left(|(E31)|+|(E32)| \right) = 0.
\end{equation}
For $(E33)$, similar arguments used in estimating $(B21)$ can be reused
\begin{align*}
	|(E33)| &\leq \eta^{-1}\left|\Binttt (\vre\ue - (\vr\uu)^\eps)\di\left(\phi'\left(\frac{\vre}{\eta}\right)\frac{|(\vr\uu)^\eps|^2}{\vre} \right)\dxte\right|\\
	&\quad + \eta^{-1}\left|\Binttt \phi'\left(\frac{\vre}{\eta} \right)\frac{\partial_t\vre}{\vre}|(\vr\uu)^\eps|^2\dxte \right|\\
	&\quad + \eta^{-1}\left|\BintttB (\vre\ue - (\vr\uu)^\eps)\phi'\left(\frac{\vre}{\eta} \right)\frac{|(\vr\uu)^\eps|^2}{\vre}\dHte \right|\\
	&=: (E331) + (E332) + (E333).
\end{align*}
The terms $(E331)$ and $(E332)$ are estimated similarly to $(B211)$ and $(B212)$, respectively,
\begin{equation*}
	|(E331)| \leq C(M_{\eps_1})\eta^{-3}\|\uu\|_{L^3}\|\uu\|_{L^3(0,T;\V_{\eps}^{\frac 13,3}(\Omega_{\frac{\eps_1}{2}}))}(\|\uu\|_{L^3(0,T;\V_{\eps}^{\frac 13,3}(\Omega_{\frac{\eps_1}{2}}))} + \|\uu\|_{L^3})
\end{equation*}
and
\begin{equation*}
	|(E332)| \leq C\|\partial_t\vr\mathbf{1}_{\{\vr\leq \delta_0\}}\|_{L^3}\|\uu \mathbf{1}_{\{\vr \leq \frac 83\eta\}}\|_{L^3}^2,
\end{equation*}
while the term $(E333)$ is estimated similarly to $(E1)$
\begin{equation*}
	\limsup_{\varepsilon_1;\varepsilon;\tau \to 0}|(E333)| \leq C\eta^{-2}\left( \int_0^T \fint_{\Omega \setminus \Omega_{\varepsilon_3}}   |\uu|^{3} dxds \right)^{\frac{2}{3}} \left( \int_0^T \fint_{\Omega \setminus \Omega_{\varepsilon_3}} |\uu \cdot n|^3 dxds \right)^{\frac{1}{3}}.
\end{equation*}
Combining these estimates one gets
\begin{equation}\label{k17}
	\limsup_{\eta;\eps_3;\eps_1;\eps;\tau \to 0}|(E33)| = 0.
\end{equation}
It then follows from \eqref{k15}, \eqref{k16} and \eqref{k17} that
\begin{align*}
\limsup_{\eta;\eps_3;\eps_1;\eps;\tau \to 0}|(D3)+(E3)| = 0.
\end{align*}
\subsection{Conclusion}
From the estimate of $(D)$, $(E)$ and $(F)$ we have
\begin{equation*}
\limsup_{\eta;\eps_3;\eps_1;\eps;\tau \to 0}\left|\frac 12\frac{1}{\eps_3}\int_{\eps_1}^{\eps_1+\eps_3}\int_\tau^t\int_{\Omega_{\eps_2}}\partial_t\left[\phi_{\eta,\eps}\frac{|(\vr\uu)^\eps|^2}{\vre}(x,s)\right]dxds\eps_2\right| = 0.
\end{equation*}
Arguing similarly to the case of torus $\Omega = \T^d$ we obtain finally the results of Theorem \ref{bounded},
\begin{equation*}
	\int_{\Omega}(\vr|\uu|^2)(x,t)dx = \int_{\Omega}(\vr_0|\uu_0|^2)(x)dx \quad \forall t\in (0,T).
\end{equation*}

\section{Proof of Theorem \ref{toruscom}}\label{sec:toruscom}
From the proof of Theorem \ref{torus}, we have for any $0<\tau<t<T$, 
\begin{equation}\label{h1com}
\begin{aligned}
\Tintt \frac{\phi_{\eta,\eps}}{\vre}(\vr\uu)^\eps\pa_t(\vr\uu)^\eps dxds &+ \Tintt \frac{\phi_{\eta,\eps}}{\vre}(\vr\uu)^\eps \di(\vr\uu\otimes\uu)^\eps dxds\\
&+ \Tintt \frac{\phi_{\eta,\eps}}{\vre}(\vr\uu)^\eps \na (p(\vr))^\eps dxds= 0.
\end{aligned}
\end{equation}
Let $M$ be a constant such that 
\begin{align*}
\|\varrho\|_{L^\infty}+\|\varrho\|_{L^{\infty}(0,T;\V_{\delta_0}^{\frac{1}{3},\infty}(\mathbb{T}^d))}\leq M. 
\end{align*}
In view of the proof of Theorem \ref{torus} and Remark \ref{divergence_free}, we only need to estimate the term $(B21)$ (now without the divergence free condition) and the last term on the left-hand side of \eqref{h1com}.

\medskip
For $(B21)$ we have
\begin{align*}
	|(B21)| &= \left|\eta^{-1}\int_\tau^t\int_{\T^d}\phi'\left(\frac{\vre}{\eta}\right)\na\vre\cdot \frac{\ue}{\vre}|(\vr\uu)^\eps|^2dxds\right|\\
	&\leq \eta^{-2}\|\ue|(\vr\uu)^\eps|^2\mathbf{1}_{\{\eta\leq \vre\leq 2\eta\}}\|_{L^1}\|\na\vre \mathbf{1}_{\{\eta\leq \vre\leq 2\eta\}}\|_{L^\infty}\\
	&\leq C\|\uu \mathbf{1}_{\{\frac 13 \eta \leq \vr \leq \frac 83 \eta\}}\|_{L^3}^3\|\na \vr \mathbf{1}_{\{\frac 13\eta \leq \vr \leq \frac 83\eta\}}\|_{L^\infty}.
\end{align*}
Therefore, thanks to \eqref{assumption-1},
\begin{equation*}
\limsup_{\eta;\delta;\eps;\tau \to 0}|(B21)| = 0.
\end{equation*}
It therefore remains to estimate the last term on the left-hand side of \eqref{h1com}. Put
\begin{equation} \label{p1} p_1(z) := \int_1^z\frac{p'(\ell)}{\ell}d\ell.
\end{equation}
We can then estimate 
\begin{align}\label{G}
	(G)&:= \int_\tau^t\int_{\T^d}\frac{\phi_{\eta,\eps}}{\vre}(\vr\uu)^\eps \na(p(\vr))^\eps dxds\nonumber\\
	&=\int_\tau^t\int_{\T^d} \frac{\phi_{\eta,\eps}}{\vre}[(\vr\uu)^\eps - \vre\ue] \na(p(\vr))^\eps dxds + \int_\tau^t\int_{\T^d}{\phi_{\eta,\eps}}\ue \na(p(\vr))^\eps dxds\nonumber\\
	&=: (G1) + \int_\tau^t\int_{\T^d}\phi_{\eta,\eps}\ue \na[(p(\vr))^\eps - p(\vre) ]dxds + \int_\tau^t\int_{\T^d}\phi_{\eta,\eps}\ue p'(\vre)\na \vre dxds\nonumber\\
	&=: (G1) + (G2) + \int_\tau^t\int_{\T^d}\phi_{\eta,\eps}\ue\vre \underbrace{\frac{p'(\vre)}{\vre}\na \vre}_{= \na p_1(\vre)}dxds \nonumber\\
	&= (G1) + (G2) + \int_\tau^t\int_{\T^d}\phi_{\eta,\eps}[\ue\vre - (\vr\uu)^\eps]\na p_1(\vre)dxds - \int_\tau^t\int_{\T^d}p_1(\vre)\na (\phi_{\eta,\eps}(\vr\uu)^\eps)dxds\nonumber\\
	&=: (G1) + (G2) + (G3) - \int_\tau^t\int_{\T^d}p_1(\vre)(\vr\uu)^\eps\na \phi_{\eta,\eps}dxds - \int_\tau^t\int_{\T^d}p_1(\vre)\phi_{\eta,\eps}\di(\vr\uu)^\eps dxds\nonumber\\
	&=: (G1) + (G2) + (G3) - (G4) - (G5).
\end{align}
To estimate $(G1)$, we observe from the fact $\int_{\R^d}\nabla \omega_\varepsilon(y)dy=0$ and H\"older's inequality that
$$ |\nabla (p(\varrho(x,t))^\varepsilon| \leq C\varepsilon^{-1-\frac{2d}{3}} \left( \int_{|y|<\varepsilon}|p(\varrho(x-y,t)-p(\varrho(x))|^{\frac{3}{2}}dy \right)^{\frac{2}{3}}.
$$
It follows that
\begin{align*}
&\int_{\T^d} |\nabla (p(\varrho(x,t))^\varepsilon|^{\frac{3}{2}} {\bf 1}_{ \{ \varrho \geq \eta  \}  }(x,t)dx \\
& \leq C\varepsilon^{-\frac{3}{2}-d} \int_{|y|<\varepsilon} \int_{\T^d} |p(\varrho(x-y,t)-p(\varrho(x,t))|^{\frac{3}{2}} {\bf 1}_{ \{ \varrho \geq \eta  \}  }(x,t)dx dy \\
&= C\varepsilon^{-\frac{3}{2}-d} \int_{|y|<\varepsilon} \int_{\T^d}  |p'(\xi)|^{\frac{3}{2}}|\varrho(x-y,t)-\varrho(x,t)|^{\frac{3}{2}} {\bf 1}_{ \{ \varrho \geq \eta  \}  }(x,t)dx dy,
\end{align*}
where $\xi$ is between $\varrho(x)$ and $\varrho(x-y,t)$. Since $\vr \in L^\infty(0,T;\V_\delta^{\frac 13,\infty}(\T^d))$, if $\varrho(x,t) \geq \eta$ then $\varrho(x-y,t) > \frac{2}{3}\eta$ for any $|y|<\varepsilon$ (see \eqref{B-rho1}). Therefore $\xi \in [\frac{2}{3}\eta, \| \varrho \|_{L^\infty}]$ and hence $|p'(\xi)|$ is bounded in $\{ \varrho \geq \eta \}$. Consequently,
\begin{align*}
&\int_{\T^d} |\nabla (p(\varrho(x,t))^\varepsilon|^{\frac{3}{2}} {\bf 1}_{ \{ \varrho \geq \eta  \}  }(x,t)dx \\
&\leq C(\eta)\varepsilon^{-\frac{3}{2}-d} \int_{|y|<\varepsilon} \int_{\T^d}  |\varrho(x-y,t)-\varrho(x)|^{\frac{3}{2}} {\bf 1}_{ \{ \varrho \geq \eta  \}  }(x,t)dx dy \\
&\leq C(\eta)\varepsilon^{-\frac{3}{2}-d} \int_{|y| < \varepsilon} |y|^{\frac{1}{2}} \left( |y|^{-\frac{1}{3}} \| \varrho(x-y,t) - \varrho(x,t) \|_{L^{\frac{3}{2}}(\T^d)} \right)^{\frac{3}{2}} dy. 
\end{align*}
This implies
\begin{equation} \label{nablap} \| |\nabla (p(\varrho(x,t))^\varepsilon {\bf 1}_{ \{ \varrho \geq \eta  \}} \|_{L^{\frac{3}{2}}(\T^d)} \leq C(\eta)\varepsilon^{-\frac{2}{3}} \| \varrho(t) \|_{\V^{\frac{1}{3},\frac{3}{2}}(\T^d)}.
\end{equation} 
Using \eqref{nablap} and arguments in estimating $(C1)$ in \eqref{h2_1}, we have
\begin{equation}\label{G1}
\begin{aligned}
	|(G1)| &\leq \int_\tau^t\int_{\T^d}\left|\frac{\phi_{\eta,\eps}}{\vre}\right||(\vr\uu)^\eps - \vre\ue||\na (p(\vr))^\eps|dxds\\
	&\leq C\eta^{-1} \int_0^T \|(\vr\uu)^\eps - \vre\ue\|_{L^3(\T^d)}\|\na (p(\vr))^\eps\mathbf{1}_{\{\vr\geq\eta \}}\|_{L^{\frac 32}(\T^d)}ds\\
	&\leq C(\eta)\eta^{-1} \int_0^T \eps^{\frac{2}{3}} \|\vr\|_{\V_\delta^{\frac 13,\infty}(\T^d)}\|\uu\|_{\V_\delta^{\frac 13,3}(\T^d)}\eps^{-\frac{2}{3}}\|\vr \|_{\V_{\delta}^{\frac 13,\frac 32}(\T^d)}ds\\
	&\leq C(M, \eta)\eta^{-1}\|\uu\|_{L^3(0,T;\V_\delta^{\frac 13,3}(\T^d))}.
\end{aligned}
\end{equation}
Therefore,
$$
\limsup_{\eta;\delta;\eps;\tau \to 0}|(G1)| = 0.
$$
Similarly,
\begin{align*}
|(G3)| &\leq \int_\tau^t \int_{\T^d} \phi_{\eta,\varepsilon} |\varrho^\varepsilon \uu^\varepsilon - (\varrho \uu)^\varepsilon|\frac{|p'(\varrho^\varepsilon)|}{\varrho^\varepsilon}|\nabla \varrho^\varepsilon|dxds \\
&\leq C(\eta)\eta^{-1} \int_0^T \| \varrho^\varepsilon \uu^\varepsilon - (\varrho \uu)^\varepsilon \|_{L^3(\T^d)} \| \nabla \varrho^\varepsilon \|_{L^{\frac{3}{2}}(\T^d)}ds \\
&\leq C(\eta)\eta^{-1} \int_0^T \varepsilon^{\frac{2}{3}}\| \varrho \|_{\V_\delta^{\frac{1}{3},\infty}(\T^d)} \| \uu \|_{\V_\delta^{\frac{1}{3},3}(\T^d)} \varepsilon^{-\frac{2}{3}} \| \varrho \|_{\V_\delta^{\frac{1}{3},\frac{3}{2}}(\T^d)}ds \\
&\leq C(\eta,M)\eta^{-1} \| \uu \|_{L^3(0,T;\V_\delta^{\frac{1}{3},3}(\T^d))}.
\end{align*}
Therefore
$$
\limsup_{\eta;\delta;\eps;\tau \to 0}|(G3)| = 0.
$$
For $(G4)$, we use \eqref{A12-4} and \eqref{assumption-1} to obtain
\begin{align*}
	|(G4)| &\leq \int_\tau^t\int_{\T^d}\left|p_1(\vre)(\vr\uu)^\eps \phi'\left(\frac{\vre}{\eta} \right)\frac{\na \vre}{\eta} \right|dxds\\
	&\leq C\eta^{-1}\|p_1(\vre) {\bf 1}_{ \{ \eta \leq \varrho^\varepsilon \leq 2\eta  \} } \|_{L^\infty}\|(\vr\uu)^\eps\mathbf{1}_{\{\eta\leq\vre\leq 2\eta \}}\|_{L^1}\|\na\vre\mathbf{1}_{\{\eta\leq\vre\leq 2\eta \}}\|_{L^\infty}\\
	&\leq C\|p_1(\vre) {\bf 1}_{ \{ \eta \leq \varrho^\varepsilon \leq 2\eta  \}  } \|_{L^\infty} \|\uu\mathbf{1}_{\{\frac 13\eta\leq \vr\leq \frac 83\eta \}}\|_{L^1}.
\end{align*}
We notice that 
$$ p_1(\varrho^\varepsilon)=\frac{p(\varrho^\varepsilon) + p_2(\varrho^\varepsilon)}{\varrho^\varepsilon} - p(1),$$
hence by \eqref{pressure_law}, $\|p_1(\vre) {\bf 1}_{ \{ \eta \leq \varrho^\varepsilon \leq 2\eta  \}  } \|_{L^\infty} < C$, where $C$ is independent of $\varepsilon$ and $\eta$. Therefore,
\begin{align*} |(G4)| \leq C\|\uu\mathbf{1}_{\{\frac 13\eta\leq \vr\leq \frac 83\eta \}}\|_{L^1}
\end{align*}
and consequently,
\begin{equation*}
\limsup_{\eta;\delta;\eps;\tau \to 0}|(G4)| = 0.
\end{equation*}
We use $\partial_t\vre = \di (\vr\uu)^\eps$ and the fact that
$$ \partial_t p_2(\varrho^\varepsilon) = \partial_t \varrho^\varepsilon (p_1(\varrho^\varepsilon) + p(1))
$$ 
(recalling $p_2$ defined in \eqref{p2}) to estimate $(G5)$,
\begin{align*}
	(G5) &= \int_\tau^t\int_{\T^d}\phi_{\eta,\eps}\partial_t p_2(\vre)dxds - p(1) \int_\tau^t\int_{\T^d} \phi_{\eta,\varepsilon} \partial_t \varrho^\varepsilon dxds \\
	&= \int_{\tau}^t\int_{\T^d}\partial_t(\phi_{\eta,\eps}p_2(\vre))(x,s)dxds - \int_\tau^t\int_{\T^d}p_2(\vre)\partial_t\phi_{\eta,\eps}dxds - p(1) \int_\tau^t\int_{\T^d} \phi_{\eta,\varepsilon} \partial_t \varrho^\varepsilon dxds\\
	&=: (G51) - (G52) - (G53).
\end{align*}
Since $(G51)$ is a desired term, we only need to estimate $(G52)$ and $(G53)$. We have
\begin{align*}
	|(G52)| &\leq \eta^{-1}\int_\tau^t\int_{\T^d}\left|p_2(\vre)\phi'\left(\frac{\vre}{\eta}\right)\partial_t\vre \right|dxds \\
	&\leq C\left\|\Bigl|\frac{p_2(\vre)}{\vre}\Bigr| | \mathbf{1}_{\{\eta\leq\vre\leq 2\eta \}}\partial_t\vre|\right\|_{L^1}\\
	&\leq C\| \mathbf{1}_{\{\eta\leq\vre\leq 2\eta \}}\partial_t\vre\|_{L^1}\\
	& \leq C\| \mathbf{1}_{\{\eta\leq \vr\leq 2\eta\}}\partial_t\vre\|_{L^3} |\{ (x,t) \in Q_T: \eta \leq \varrho^\varepsilon \leq 2\eta \}|^{\frac{2}{3}} \\
	&\leq CM|\{ (x,t) \in Q_T: \eta/3 \leq \varrho \leq 8\eta/3 \}|^{\frac{2}{3}}.
\end{align*}
for small enough $\eta>0$. Here we have used \eqref{pressure_law} for the third estimate, H\"older's inequality for the fourth estimate and \eqref{A12-3} for the last estimate. Thus, thanks to \eqref{assumption-1}, 
\begin{equation*}
\limsup_{\eta;\delta;\eps;\tau \to 0}|(G52)| = 0.
\end{equation*}

Next we estimate $(G53)$. By integration by parts,
\begin{align*}
(G53) &= p(1)\int_\tau^t\int_{\T^d} \phi_{\eta,\varepsilon} \partial_t \varrho^\varepsilon dxds \\ &=p(1) \int_{\tau}^t \int_{\T^d} (\varrho \uu)^\varepsilon \nabla \phi_{\eta,\varepsilon}dxds \\
&= \eta^{-1} p(1) \int_{\tau}^t \int_{\T^d} (\varrho \uu)^\varepsilon \phi' \left( \frac{\varrho^\varepsilon}{\eta} \right) \nabla \varrho^\varepsilon {\bf 1}_{ \{ \eta \leq  \varrho^\varepsilon \leq 2\eta \} } dxds
\end{align*}
Therefore, by H\"older inequality, \eqref{A12-4} and \eqref{assumption-1},
\begin{align*}
|(G53)| &\leq C\eta^{-1} \int_\tau^t \int_{\T^d} |(\varrho \uu)^\varepsilon| |\nabla \varrho^\varepsilon| {\bf 1}_{ \{ \eta \leq \varrho^\varepsilon \leq 2\eta \} } dxds \\
& \leq C\eta^{-1}  \|  (\varrho \uu)^\varepsilon {\bf 1}_{ \{ \eta \leq \varrho^\varepsilon \leq 2\eta \} }  \|_{L^3}  \| \nabla \varrho^\varepsilon {\bf 1}_{ \{ \eta\leq  \varrho^\varepsilon \leq 2\eta \} } \|_{L^{\frac{3}{2}}} \\
&\leq CM \| \uu {\bf 1}_{ \{ \frac{1}{3}\eta \leq \varrho \leq \frac{8}{3}\eta \} }\|_{L^3} \| \nabla \varrho {\bf 1}_{ \{  \varrho \leq \delta_0 \} }   \|_{L^{\frac{3}{2}}} \\
&\leq CM \| \uu {\bf 1}_{ \{ \frac{1}{3}\eta \leq \varrho \leq \frac{8}{3}\eta \} }\|_{L^3}.
\end{align*}
It follows that
$$ \limsup_{\eta;\delta;\varepsilon;\tau \to 0} |(G53)|=0.
$$

It therefore remains to estimate $(G2)$. Integration by parts leads to
\begin{align*}
	(G2) &= -\int_\tau^t\int_{\T^d}\na\phi_{\eta,\eps} \ue [(p(\vr))^\eps - p(\vre)]dxds - \int_\tau^t\int_{\T^d}\phi_{\eta,\eps}\na\ue [(p(\vr))^\eps - p(\vre)]dxds\\
	&=: (G21) + (G22).
\end{align*}
We will show that
\begin{equation}\label{pressure}
	\|((p(\vr))^\eps - p(\vre))\mathbf{1}_{\{\eta\leq \vre\}}\|_{L^{\frac 32}(\T^d)} \leq C(\eta)\eps
\end{equation}
where the constant $C(\eta)$ might blow up as $\eta \to 0$.
Indeed, first, thanks to Taylor's expansion and the continuity of $p''$, we have
\begin{equation}\label{p''}
	|p(a+b) - p(a) - p'(a)b| = |p''(\xi)||b|^2 \leq \Lambda(\eta)|b|^2
\end{equation}
for all $a, b$ such that $\frac 13\eta \leq a, a+b \leq \|\vr\|_{L^\infty}$, where $\Lambda(\eta) = \sup\{|p''(\xi)|: \frac 13\eta \leq \xi \leq \|\vr\|_{L^\infty} \}$. Now by writing
\[
	\vr(x-y,s) = \vr(x,s) + [\vr(x-y,s) - \vr(x,s)]
\]
and
\[
	\int_{\T^d}\vr(x-y,s)\omega_{\eps}(y)dy = \vr(x,s) + \int_{\T^d}(\vr(x-y,s) - \vr(x,s))\omega_\eps(y)dy
\]
we can apply \eqref{p''} with $a = \vr(x,s)$ and $b = \vr(x-y,s) - \vr(x,s)$, and then $a = \vr(x,s)$ and $b = \int_{\T^d}(\vr(x-y,s) - \vr(x,s))\omega_\eps(y)dy$ respectively, to obtain
\begin{align*}
	&|(p(\vr))^\eps(x,s) - p(\vre)(x,s)|\mathbf{1}_{\{\frac 13 \eta \leq \vr\}}(x,s)\\
	&\leq \left|\int_{\T^d}[p(\vr(x-y,s)) - p(\vr(x,s)) - p'(\vr(x,s))(\vr(x-y,s) - \vr(x,s))]\omega_\eps(y)dy\right|\mathbf{1}_{\{\frac 13 \eta \leq \vr\}}(x,s)\\
	&+ \left|p\left(\int_{\T^d}\vr(x-y,s)\omega_\eps(y)dy \right) - p(\vr(x,s)) - p'(\vr(x,s))\int_{\T^d}(\vr(x-y,s) - \vr(x,s))\omega_\eps(y)dy \right|\mathbf{1}_{\{\frac 13 \eta \leq \vr\}}(x,s)\\
	&\leq 2\Lambda(\eta)\int_{\T^d}|\vr(x-y,s) - \vr(x,s)|^2\omega_\eps(y)dy\\
	&\leq C(\eta) \varepsilon^{\frac{2}{3}}.
\end{align*}
Therefore, 
$$
	\|((p(\vr))^\eps(\cdot,s) - p(\vre)(\cdot,s))\mathbf{1}_{\{\eta\leq \vre\}}\|_{L^{\frac 32}(\T^d)}\leq C(\eta)\varepsilon^{\frac{2}{3}}.
$$
Using \eqref{pressure} we now can esimtate
\begin{align*}
	|(G21)| &\leq \eta^{-1}\int_{\tau}^t\int_{\T^d}\left|\phi'\left(\frac{\vre}{\eta}\right)\na \vre \ue [(p(\vr))^\eps - p(\vre)]\right|dxds \\
	&\leq C\eta^{-1}\int_0^T\|\na \vre\|_{L^\infty(\T^d)}\|\ue\|_{L^3(\T^d)}\|((p(\vr))^\eps - p(\vre))\mathbf{1}_{\{\eta \leq \vre\}}\|_{L^{\frac 32}(\T^d)}ds\\
	&\leq C\eta^{-1}\eps^{-\frac 23}\|\vr\|_{L^\infty(0,T;\V_\delta^{\frac 13,\infty}(\T^d))}\|\uu\|_{L^3}C(\eta)\varepsilon^{\frac{2}{3}}\\
	&\leq C(\eta)\eta^{-1} \|\vr\|_{L^\infty(0,T;\V_\delta^{\frac 13,\infty}(\T^d))}\|\uu\|_{L^3}
\end{align*}
and thus by \eqref{Tucom},
\begin{equation*}
\limsup_{\eta;\delta;\eps;\tau \to 0}|(G21)| = 0.
\end{equation*}
Similarly
\begin{align*}
	|(G22)| &\leq \int_{\tau}^t\int_{\T^d}|\na \ue||\mathbf{1}_{\{\eta\leq\vre\}}[(p(\vr))^\eps - p(\vre)]|dxds\\
	&\leq \int_0^T \|\na\ue\|_{L^3(\T^d)}\|\mathbf{1}_{\{\eta\leq\vre\}}[(p(\vr))^\eps - p(\vre)]\|_{L^{\frac 32}(\T^d)}ds\\
	&\leq C\eps^{-\frac 23}\|\uu\|_{L^3(0,T;\V_\delta^{\frac 13,3}(\T^d))}C\Lambda(\eta)\eps^{\frac{2}{3}}\\
	&\leq C\Lambda(\eta)\|\uu\|_{L^3(0,T;\V_\delta^{\frac 13,3}(\T^d))},
\end{align*}
hence by \eqref{Tucom},
\begin{equation*}
	\limsup_{\eta;\delta;\eps;\tau \to 0}|(G22)| = 0.
\end{equation*}
\subsection{Conclusion}
From the previous estimates we have 
\begin{equation*}
\limsup_{\eta;\delta;\eps;\tau \to 0}\left|\Tintt \pa_t\left[\frac 12\frac{\phi_{\eta,\eps}}{\vre}|(\vr\uu)^\eps|^2+\phi_{\eta,\eps}p_2(\vre)\right]dxds\right| = 0.
\end{equation*}
By proceeding as in Section \ref{conclusion} with an additional argument concerning the term $p_2(\vr)$, we derive the energy conservation  \eqref{conserEc}. \qeda

\section{Proof of Theorem \ref{boundedcom}}\label{sec:boundedcom}
The proof of Theorem \ref{boundedcom} is similar to that of Theorem \ref{toruscom}, except that we have to take care of boundary layers when taking integration by parts. 
Take $0 < \eps < \eps_1/10 < \eps_3/10<\delta_0/100$, as proof of Theorem \ref{bounded}, we have
\begin{equation}\label{k2com}
\begin{aligned}
\frac{1}{\varepsilon_3} \int_{\varepsilon_1}^{\varepsilon_1 + \varepsilon_3}  \int_{\tau}^t \int_{\Omega_{\varepsilon_2}} \frac{\phi_{\eta,\eps}}{\vre}(\vr\uu)^\eps\pa_t(\vr\uu)^\eps dxdsd\varepsilon_2 &+ \frac{1}{\varepsilon_3} \int_{\varepsilon_1}^{\varepsilon_1 + \varepsilon_3}  \int_{\tau}^t \int_{\Omega_{\varepsilon_2}} \frac{\phi_{\eta,\eps}}{\vre}(\vr\uu)^\eps \di(\vr\uu\otimes\uu)^\eps dxdsd\varepsilon_2\\
&+ \frac{1}{\varepsilon_3} \int_{\varepsilon_1}^{\varepsilon_1 + \varepsilon_3}  \int_{\tau}^t \int_{\Omega_{\varepsilon_2}} \frac{\phi_{\eta,\eps}}{\vre}(\vr\uu)^\eps \na (p(\vr))^\eps dxdsd\varepsilon_2 = 0.
\end{aligned}
\end{equation}
In view of the proof of Theorem \ref{bounded}, we only have to deal with the term $(E33)$ (now without the divergence free condition) and the last term on the left-hand side of \eqref{k2com}.

\medskip
For $(E33)$ we have, using arguments similar to \eqref{A12-4}, 
\begin{align*}
	|(E33)| &= \frac 12\eta^{-1}\Binttt |(\vr\uu)^\eps|^2\phi'\left(\frac{\vre}{\eta} \right)\na \vre \frac{\ue}{\vre}\dxte\\
	&\leq C\eta^{-2}\Binttt|\ue||(\vr\uu)^\eps|^2|\mathbf{1}_{\eta \leq \vre}\na\vre|\dxte\\
	&\leq C\|\mathbf{1}_{\{\frac 13\eta\leq \vr \leq \frac 83\eta\}}\na \vre\|_{L^\infty}\left\|\uu\mathbf{1}_{\{\frac 13\eta\leq \vr \leq \frac 83\eta\}}\right\|_{L^3}^3,
\end{align*}
and therefore
\begin{equation*}
	\limsup_{\eta;\eps_3;\eps_1;\eps;\tau \to 0}|(E33)| = 0
\end{equation*}
thanks to \eqref{assumption-2}. For the last term on the left-hand side of \eqref{k2com}, we we can repeat all the arguments in estimating $(G)$ in the proof of Theorem \ref{toruscom}.  The only extra effort is the boundary terms when integrating by parts. More precisely, for the sixth line of \eqref{G}, $(G2)$ and $(G53)$, which are
\begin{equation*}
	(H1) = \BintttB p_1(\vre)\phi_{\eta,\eps}(\vr\uu)^\eps \cdot n(\theta)\dHte,
\end{equation*}
\begin{equation*}
	(H2) = \BintttB \phi_{\eta,\eps} \ue[(p(\vr))^\eps - p(\vre)]\cdot n(\theta)\dHte,
\end{equation*}
and
\begin{equation*}
	(H3) = p(1)\BintttB \phi_{\eta,\eps}(\vr\uu)^\eps \cdot n(\theta)\dHte,
\end{equation*}
respectively. Using the assumptions \eqref{uboundarycom} and \eqref{Pboundarycom} we can easily estimate $(H1)$, $(H2)$ and $(H3)$ similarly to $(E31)$ and $(E32)$ to obtain
\begin{equation*}
\limsup_{\eta;\eps_3;\eps_1;\eps;\tau \to 0}(|(H1)| + |(H2)| + |(H3)|) = 0.
\end{equation*}
Therefore, we conclude 
\begin{align*}
  \limsup_{\eta;\varepsilon_3;\varepsilon_1;\varepsilon;\tau \to 0}   \left|\frac{1}{\varepsilon_3} \int_{\varepsilon_1}^{\varepsilon_1 + \varepsilon_3}\!\!\!\int_{\tau}^t \int_{\Omega_{\varepsilon_2}} \pa_t\left[\frac 12\frac{\phi_{\eta,\eps}}{\vre}|(\vr\uu)^\eps|^2+\phi_{\eta,\eps}p_2(\vre)\right]dxds \right|=0.
\end{align*}
This gives the energy conservation \eqref{conserEcboun}. The proof is complete. \qeda
\appendix \section{Appendix}\label{appendix}
In this section, we collect some lemmata and estimates which are used for the proofs of the main results.

We recall that, for a function $f:\Omega\to \mathbb R$ we denote by $f^\eps = f\star \omega_\eps$ the smoothing version of $f$, where $\omega_\eps(x) = (1/\eps^d)\omega(x/\eps)$ is a standard mollifier in $\mathbb R^d$.

\begin{lemma} \label{keylemma-boundeddomain}
	Let $\Omega\subset \mathbb R^d$ be a bounded domain with $C^2$ boundary $\pa\Omega$. 
	\begin{itemize}
		\item[(1)] Let $\beta\in (0,1)$ and $1 \leq p \leq \infty$. Then for any function $f: \Omega \to {\mathbb R}$ and small $0<\varepsilon\leq \frac{\delta}{2}$, there holds
		\begin{align} \label{fe1}
		&	\|\nabla f^\varepsilon\|_{L^p(\Omega_\delta)}\leq C \varepsilon^{-1}   \|f\|_{L^{p}(\Omega_{\frac{\delta}{2}} )},\\ \label{fe2}
		&\|\nabla f^\varepsilon\|_{L^p(\Omega_\delta)}\leq C \varepsilon^{-1+\beta }   \|f\|_{\mathcal{V}^{\beta,p}_{\delta}(\Omega)}.
		\end{align}
		\item[(2)] Let $\beta_1,\beta_2\in (0,1)$, $p\geq 1, p_1,p_2\geq 1$ such that $\frac{1}{p}=\frac{1}{p_1}+\frac{1}{p_2}$ and $\delta_1>0, \delta_2>0, \delta \geq \max\{ \delta_1, \delta_2 \}$. Then for any functions $g_1, g_2: \Omega \to {\mathbb R}$ and small $0<\varepsilon\leq \frac{\min \{ \delta_1,\delta_2 \} }{2}$, there holds
		\begin{equation} \label{fe3}
		\|(g_1g_2)^\varepsilon-g_1^\varepsilon g_2^\varepsilon\|_{L^p(\Omega_\delta)}\leq C \varepsilon^{\beta_1+\beta_2} \|g_1\|_{\mathcal{V}^{\beta_1,p_1}_{\delta_1}(\Omega)} \|g_2\|_{\mathcal{V}^{\beta_2,p_2}_{\delta_2}(\Omega)}.
		\end{equation}
		\item[(3)] Let $\beta \in (0,1)$  and  $p \geq 1$. Then for any functions $g_1, g_2: \Omega \to {\mathbb R}$ and small $\delta\in (0,1)$, there holds
		\begin{align}\label{fe4}
		\|g_1 g_2\|_{\mathcal{V}^{\beta,p}_{\varepsilon}(\Omega_{\delta})}\leq C\left(\|g_1\|_{L^{\infty}(\Omega_{\frac{\delta}{2}})}\| g_2\|_{\mathcal{V}^{\beta,p}_{\varepsilon}(\Omega_{\frac{\delta}{2}})}+\|g_1\|_{\mathcal{V}^{\beta,\infty}_{\delta}(\Omega_{\frac{\delta}{2}})}	\| g_2\|_{L^{p}(\Omega_{\frac{\delta}{2}})}\right),
		\end{align}
		for any $0<\varepsilon<\delta/4$.
	\end{itemize}
\end{lemma}
\begin{proof}
	1) Since $\int_{\mathbb{R}^d}\nabla \ox_\varepsilon(y)dy=0$, it follows that, for a.e. $x \in \Omega_{\delta}$,  
	\begin{align*}
	|\nabla f^\varepsilon(x)|&= \Big|\int_{\mathbb{R}^d}\left[f(x-y)-f(x)\right]\nabla \ox_\varepsilon(y)dy\Big|\leq \varepsilon^{-1-\frac{d}{p}}  \left(\int_{|y|<\varepsilon}\left|f(x-y)-f(x)\right|^{p}dy\right)^{\frac{1}{p}} \|\nabla\ox\|_{L^{\frac{p}{p-1}}}.
	\end{align*}
	This yields \eqref{fe1} and \eqref{fe2}.\smallskip
%
	
	2) It is not hard to see that 
	\begin{align*}
	|(g_1g_2)^\varepsilon(x)-g_1^\varepsilon(x)g_2^\varepsilon(x)|&\leq \int_{\mathbb{R}^d}|g_1(x-y)-g_1(x)\|g_2(x-y)-g_2(x)|\omega_\varepsilon(y) dy\\&+|g_1^\varepsilon(x) -g_1(x)\|g_2^\varepsilon(x)- g_2(x)|.
	\end{align*}
Thus, using Holder's inequality, we get \eqref{fe3}. \smallskip

	3) For every $x \in \Omega_\delta$ and $h \in \R^d$ such that $|h|<\delta/2$, we have
	$$ |(fg)(x+h)-(fg)(x)| \leq |f(x+h)| |g(x+h)-g(x)| + |f(x+h)-f(x)| |g(x)|.
	$$ 	
	Clearly, this gives \eqref{fe4}. 
%
%
\end{proof}

\begin{lemma} \label{keylemma-Td} 
	~~ \smallskip	
	
	\begin{itemize}
		\item[(1)] Let $\beta\in (0,1)$ and $1 \leq p \leq \infty$. Then for any function $f: \mathbb{T}^d \to {\mathbb R}$ and  $0<\varepsilon < \delta$, there holds
		\begin{align*}
		&	\|\nabla f^\varepsilon\|_{L^p(\mathbb{T}^d)}\leq C \varepsilon^{-1}   \|f\|_{L^{p}(\mathbb{T}^d)},\\&
		\|\nabla f^\varepsilon\|_{L^p(\mathbb{T}^d)}\leq C  \varepsilon^{-1+\beta }   \|f\|_{\mathcal{V}^{\beta,p}_{\delta}(\mathbb{T}^d)}.
		\end{align*}
		\item[(2)] Let $\beta_1,\beta_2\in (0,1)$ and  $p\geq 1, p_1,p_2\geq 1$ such that $\frac{1}{p}=\frac{1}{p_1}+\frac{1}{p_2}$. Then for any functions $g_1, g_2: \mathbb{T}^d \to {\mathbb R}$ and  $0<\varepsilon<\delta$, there holds
		\begin{align*}
		&\|(g_1g_2)^\varepsilon-g_1^\varepsilon g_2^\varepsilon\|_{L^p(\mathbb{T}^d)}\leq C \varepsilon^{\beta_1+\beta_2} \|\ox\|_{L^{\frac{p}{p-1}}} \|g_1\|_{\mathcal{V}^{\beta_1,p_1}_{\delta}(\mathbb{T}^d)} \|g_2\|_{\mathcal{V}^{\beta_2,p_2}_{\delta}(\mathbb{T}^d)}.
		\end{align*}
		\item[(3)] Let $\beta \in (0,1)$  and  $p \geq 1$. Then for any functions $g_1, g_2: \mathbb{T}^d \to {\mathbb R}$ and $\varepsilon>0$, there holds
		\begin{align*}
		\|g_1 g_2\|_{\mathcal{V}^{\beta,p}_{\varepsilon}(\mathbb{T}^d)}\leq C\left(\|g_1\|_{L^{\infty}(\mathbb{T}^d)}\| g_2\|_{\mathcal{V}^{\beta,p}_{\varepsilon}(\mathbb{T}^d)}+\|g_1\|_{\mathcal{V}^{\beta,\infty}_{\varepsilon}(\mathbb{T}^d)}	\| g_2\|_{L^{p}(\mathbb{T}^d)}\right).
		\end{align*}
	\end{itemize}
\end{lemma}
\begin{proof}
	The proof is similar to that of Lemma \ref{keylemma-boundeddomain} and we omit it.	
\end{proof}

\begin{lemma} \label{phi} Assume that there exists $\delta_0>0$ and $M>0$ such that
	\begin{equation} \label{Besovrho}
		\|\varrho\|_{L^\infty(0,T;\V_{\delta_0}^{\frac{1}{3},\infty}(\mathbb{T}^d))}<M.
	\end{equation}	
	For any $0<\varepsilon^{\frac{1}{3}} < \frac{1}{3M}\eta <\frac{1}{12M} \delta_0$, the followings hold.
	
	(i) If $\varrho^\varepsilon(x,t) \geq \eta $ then $\varrho(x,t) > \frac{2}{3}\eta$. In addition, for any $y \in B_{\varepsilon}(0)$, $\varrho(x-y,t)> \frac{1}{3}\eta$. 
	
	(ii) If $\varrho^\varepsilon(x,t) \leq 2\eta $ then $\varrho(x,t) < \frac{7}{3}\eta$. In addition, for any $y \in B_{\varepsilon}(0)$, $\varrho(x-y,t)< \frac{8}{3}\eta$.
\end{lemma}
\begin{proof}
From \eqref{Besovrho}, we see that
\begin{equation} \label{B-rho1}
|\varrho(x-y,t)-\varrho(x,t)| < M|y|^{\frac{1}{3}} \quad \text{for a.e. } x \in \T^d,\, y \in B_{\delta_0}(0), \; t \in (0,T). 
\end{equation}	
We write
\begin{align*}
\varrho^\varepsilon(x,t)=\int_{B_\varepsilon(0)} (\varrho(x-y,t) - \varrho(x,t)) \omega_\varepsilon(y)dy + \varrho(x,t).
\end{align*}
Therefore, if $\varrho^{\varepsilon}(x,t) \geq \eta$, by \eqref{B-rho1}, we obtain 
\begin{align*} \varrho(x,t) &\geq \varrho^{\varepsilon}(x,t) - \int_{B_\varepsilon(0)} |\varrho(x-y,t) - \varrho(x,t)| \omega_\varepsilon(y)dy \geq \eta - M \varepsilon^{\frac{1}{3}} > \frac{2}{3}\eta.
\end{align*}
This, together with \eqref{B-rho1}, implies 
$$ \varrho(x-y,t) > \varrho(x,t) - M\varepsilon^{\frac{1}{3}}>\frac{1}{3}\eta.
$$
Similarly, if $\varrho(x,t) \leq 2\eta$, by \eqref{B-rho1}, we obtain 
$$ \varrho(x,t) \leq \varrho^{\varepsilon}(x,t) + \int_{B_\varepsilon(0)} |\varrho(x-y,t) - \varrho(x,t)| \omega_\varepsilon(y)dy< 2\eta + M\varepsilon^{\frac{1}{3}} < \frac{7}{3}\eta.
$$
This, together with \eqref{B-rho1}, implies
$$ \varrho(x-y,t) \leq \varrho(x,t) + M\varepsilon^{\frac{1}{3}}<\frac{8}{3}\eta. 
$$
\end{proof}

\medskip
 \par{\bf Acknowledgements:} Q.-H. Nguyen has been supported by the Centro De Giorgi, Scuola Normale Superiore, Pisa, Italy. P.-T. Nguyen has been supported by Czech Science Foundation, project GJ19 -- 14413Y. B. Q. Tang gratefully acknowledges the support of the Hausdorff Research Institute for Mathematics (Bonn), through the Junior Trimester Program on Kinetic Theory. This work is partially supported by the International Training Program IGDK 1754 and NAWI Graz.

\end{document}